\documentclass[11pt]{article}
\usepackage{amsfonts}
\usepackage{mathrsfs}
\usepackage{bbm}
\usepackage{amsfonts}
\usepackage{amssymb,amsmath,amsthm,graphicx}
\usepackage{float}
\usepackage[colorlinks, linkcolor=red]{hyperref}
\usepackage{color, xcolor}
\usepackage{cite}
\usepackage[paper=a4paper, left=1.6cm, right=1.6cm, top=1.8cm, bottom=1.6cm, headheight=5.5pt, footskip=0.8cm, footnotesep=0.8cm, centering, includefoot]{geometry}

\hypersetup{urlcolor=red, citecolor=red}

\DeclareMathOperator{\divv}{div}
\DeclareMathOperator{\curl}{curl}

\allowdisplaybreaks

\begin{document}
\title{Global well-posedness for three-dimensional compressible viscous micropolar and heat-conducting fluids with vacuum at infinity and large oscillations
\thanks{
Y. Liu was supported by National Natural Science Foundation of China (No. 11901288), Scientific Research Foundation of Jilin Provincial Education Department (No. JJKH20210873KJ), Postdoctoral Science Foundation of
China (No. 2021M691219), and Natural Science Foundation of Changchun Normal University.  X. Zhong was supported by National Natural Science Foundation of China (Nos. 11901474, 12071359) and Exceptional Young Talents Project of Chongqing Talent (No. cstc2021ycjh-bgzxm0153).}
}
\author{Yang Liu$\,^{\rm 1, 2}\,$\quad Xin Zhong$\,^{\rm 3}\,${\thanks{Corresponding author. E-mail address: liuyang0405@ccsfu.edu.cn (Y. Liu), xzhong1014@amss.ac.cn (X. Zhong).}}
\date{}\\
\footnotesize $^{\rm 1}\,$
School of Mathematics, Jilin University, Changchun 130012, P. R. China\\
\footnotesize $^{\rm 2}\,$ College of Mathematics, Changchun Normal
University, Changchun 130032, P. R. China\\
\footnotesize $^{\rm 3}\,$ School of Mathematics and Statistics, Southwest University, Chongqing 400715, P. R. China} \maketitle
\newtheorem{theorem}{Theorem}[section]
\newtheorem{definition}{Definition}[section]
\newtheorem{lemma}{Lemma}[section]
\newtheorem{proposition}{Proposition}[section]
\newtheorem{corollary}{Corollary}[section]
\newtheorem{remark}{Remark}[section]
\renewcommand{\theequation}{\thesection.\arabic{equation}}
\catcode`@=11 \@addtoreset{equation}{section} \catcode`@=12
\maketitle{}

\begin{abstract}
We investigate global well-posedness to the Cauchy problem of three-dimensional compressible viscous and heat-conducting micropolar fluid equations with zero density at infinity. By delicate energy estimates, we establish global existence and uniqueness of strong solutions under some smallness condition depending only on the parameters appeared in the system and the initial mass. In particular, the initial
mass can be arbitrarily large. This improves our previous work \cite{LZ22}. Moreover, we also generalize the result \cite{HLZ21} to the case that vacuum is allowed at infinity.
\end{abstract}

\textit{Key words and phrases}. Compressible heat-conducting micropolar fluids; global well-posedness; Cauchy problem; vacuum at infinity.

2020 \textit{Mathematics Subject Classification}.  35Q35; 76N10.


\section{Introduction}

Micropolar fluid equations, which were suggested and introduced by Eringen in the 1960s (see \cite{E66}), are a significant step toward generalization of the Navier-Stokes equations. It is a type of fluids which exhibits micro-rotational effects and micro-rotational inertia, and can be viewed as a non-Newtonian fluid. Physically, micropolar fluid may represent fluids that consist of rigid, randomly oriented (or spherical particles) suspended in a viscous medium, where the deformation of fluid particles is ignored. It can describe many phenomena that appear in a large number of complex fluids such as the suspensions, animal blood, liquid crystals which cannot be characterized appropriately by the Navier-Stokes system, and that it is important to the scientists working with the hydrodynamic-fluid problems and phenomena. We refer the reader to the monograph \cite{L1999}, which provides a detailed derivation of the micropolar fluid equations from the general constitutive laws, together with an extensive review of the mathematical theory and the applications of this particular model.
In the present paper, we consider the Cauchy problem of three-dimensional (3D for short) compressible viscous and heat-conducting micropolar fluid equations (see \cite[Chapter 1]{L1999})
\begin{align}\label{a1}
\begin{cases}
\rho_t+\divv(\rho u)=0,\\
\rho (u_t+u\cdot\nabla u)=(\lambda+\mu-\mu_r)\nabla\divv u+(\mu+\mu_r)\Delta u-\nabla p+2\mu_r\curl w,\\
j_I\rho(w_t+u\cdot\nabla w)=2\mu_r(\curl u-2w)+(c_0+c_d-c_a)\nabla\divv w+(c_a+c_d)\Delta w,\\
c_v\rho(\theta_t+u\cdot\nabla\theta)=-p\divv u+\kappa\Delta\theta+\mathcal{Q}(\nabla u)+4\mu_r\big|\frac12\curl u-w\big|^2
+c_0(\divv w)^2\\
\qquad\qquad\qquad\qquad+(c_a+c_d)\nabla w:\nabla w^T+(c_d-c_a)\nabla w:\nabla w,
\end{cases}
\end{align}
with the initial condition
\begin{align}\label{a2}
(\rho, u, w, \theta)(x, 0)=(\rho_0, u_0, w_0, \theta_0)(x), \ x\in\mathbb{R}^3,
\end{align}
and the far field behavior
\begin{align}\label{a3}
(\rho, u, w, \theta)(x, t)\rightarrow (0, 0, 0, 0)\ {\rm as}\ |x|\rightarrow \infty.
\end{align}
Here $\rho=\rho(x, t)$, $u=(u^1, u^2, u^3)(x, t)$, $w=(w^1, w^2, w^3)(x, t)$, $\theta=\theta(x, t)$ are
the density, velocity, micro-rotation velocity, and absolute temperature, respectively; $p=R\rho\theta$, with positive constant $R$, is the pressure, and
\begin{align*}
\mathcal{Q}(\nabla u)=\frac{\mu}{2}|\nabla u+(\nabla u)^\top|^2
+\lambda(\divv u)^2,
\end{align*}
with $(\nabla u)^\top$ being the transpose of $\nabla u$.
The constants $c_v$, $\kappa$, and $j_I$ are all positive. $\mu$ and $\lambda$ are the coefficients of viscosity, while $\mu_r$, $c_0$, $c_d$, and $c_a$ are the coefficients of micro-viscosity, and they satisfy
\begin{align}\label{1.4}
\mu>0,\ 3\lambda+2\mu\ge 0, \ \mu_r>0,\
c_d>0,\ 3c_0+2c_d\ge 0,\ c_a>0.
\end{align}

It should be noted that the system \eqref{a1} reduces to the full compressible Navier-Stokes equations when the rotation effect of micro-particles $w$ is ignored, which have been discussed in numerous studies on the existence, uniqueness, and regularity of solutions. Feireisl \cite{F2004} proved the global existence of the so-called ``variational solutions" with a temperature dependent coefficient of heat conduction in the sense that the energy equation is replaced by an energy inequality. Huang and Li \cite{HL18} derived global well-posedness of strong solutions in $\mathbb{R}^3$ with non-vacuum at infinity which are of small energy but possibly large oscillations. Later on, Wen and Zhu \cite{WZ17} improved the result \cite{HL18} to the case of far-field vacuum provided that the initial mass is properly small in certain sense. Meanwhile, Li \cite{JL20} obtained a new type of global strong solutions under some smallness condition on the scaling invariant quantity. Liang \cite{ZL21} proved global existence and decay rates of strong solutions when the initial energy is small enough.

Let's turn our attention to the compressible heat-conducting micropolar fluid equations \eqref{a1}. Global well-posedness of classical solutions in one dimension with arbitrary large initial data was obtained by Feng and Zhu \cite{FZ19} for the case that the coefficient of heat conduction is a temperature-dependent function, and thereafter extended by Zhang and Zhu \cite{ZZ20} to the Robin boundary conditions case,
see also \cite{D18} for some related result. We should point out that the temperature-dependent coefficient of heat conduction plays a crucial role in their proof. Recently, Wan and Zhang \cite{WZ20} studied 1D initial-boundary value problem with constant coefficients and vacuum. Based on Lagrangian coordinates and the energy method, they showed the global existence of a unique solution in $H^2$. For 3D case, Liu and Zhang \cite{LZ20} deduced the global existence and $L^2$-decay rate of classical solutions for small perturbed initial data around the equilibrium sates in some Sobolev spaces of high order. Later on, Liu-Huang-Zhang \cite{LHZ21} established global well-posedness and optimal convergence rates of strong solutions with zero heat conductivity for the initial data close to the constant equilibrium state in $H^3$ and bounded in $L^1$. Wu and Jiang \cite{WJ21} obtained the pointwise space-time estimates of classical solutions through the analysis of the Green's function. It should be mentioned that the results obtained in \cite{LZ20,WJ21,LHZ21} require that the solution has small oscillations from a uniform non-vacuum state so that the density is strictly away from vacuum and the gradient of the density remains bounded uniformly in time. on the other hand, when vacuum is allowed, some new challenging difficulties arise, such as degeneracy of the system. Nevertheless, some important progress on the global existence of strong solutions has been achieved recently. More precisely, Huang-Liu-Zhang \cite{HLZ21} investigated the global existence and uniqueness of strong solutions to the Cauchy problem for \eqref{a1} in $\mathbb{R}^3$ provided that the initial energy is properly small. Vacuum is allowed interiorly but not at infinity in \cite{HLZ21}. Very recently, Liu and Zhong \cite{LZ22} extended it to the case of allowing vacuum at infinity by requiring the initial mass small enough. The aim of the present paper is to improve the result \cite{LZ22}. This is a nontrivial generalization, since our main result shows that the strong solution exists globally in time if some smallness assumption independent of any norms of the initial data except $\|\rho_0\|_{L^1}$ holds true. This is exactly the new point of this paper. There are also some interesting mathematical results concerning the compressible heat-conducting micropolar fluid equations with spherical or cylindrical symmetry, please refer to \cite{D17,D181,D182,DM15,DMC17,DSM16,HD18,HD19,MSD17}.

For $1\le p\le \infty$ and integer $k\ge 0$, we denote the standard homogeneous and inhomogeneous
Sobolev spaces as follows:
\begin{align*}
\begin{cases}
L^p=L^p(\mathbb{R}^3),~ ~ W^{k, p}=L^p\cap D^{k, p},~~ H^k=W^{k, 2}, \\
  D^{k, p}=\{u\in L_{loc}^1(\mathbb{R}^3): \|\nabla^ku\|_{L^p}<\infty\}, ~D^k=D^{k, 2},\\
  D_0^1=\{u\in L^6(\mathbb{R}^3): \|\nabla u\|_{L^2}<\infty\}.
\end{cases}
\end{align*}

We can now state our main result.
\begin{theorem}\label{thm1}
In addition to \eqref{1.4}, assume that
\begin{align}\label{1.5}
2\mu>\lambda+\frac{\mu_r}{2},\ 2c_d>c_0+\frac{c_a}{2}.
\end{align}
For $q\in (3, 6)$, let the initial data $(\rho_0, u_0, w_0, \theta_0\ge 0)$ satisfy
\begin{align}\label{1.6}
0\leq\rho_0\le \bar{\rho},~\rho_0\in L^1\cap H^1
\cap W^{1, q}, ~(u_0, \theta_0)\in D_0^1\cap D^2,\ w_0\in H^2,
\end{align}
and the compatibility condition
\begin{align}\label{1.7}
\begin{cases}
(\lambda+\mu-\mu_r)\nabla\divv u_0+(\mu+\mu_r)\Delta u_0-R\nabla(\rho_0\theta_0)+2\mu_r\curl w_0=\sqrt{\rho_0}g_1,\\
2\mu_r(\curl u_0-2w_0)+(c_0+c_d-c_a)\nabla \divv w_0+(c_a+c_d)\Delta w_0=\sqrt{\rho_0}g_2,\\
-R\rho_0\theta_0\divv u_0+\lambda(\divv u_0)^2+\frac{\mu}{2}|\nabla u_0+\nabla u_0^T|^2+4\mu_r\big|\frac12\curl u_0-w_0\big|^2\\
\quad+c_0(\divv w_0)^2+(c_a+c_d)\nabla w_0:\nabla w_0^T+(c_d-c_a)\nabla w_0:\nabla w_0+\kappa\Delta\theta_0=\sqrt{\rho_0}g_3,
\end{cases}
\end{align}
with $g_i\in L^2\ (i\in\{1,2,3\})$. Then there exists a positive
constant $\varepsilon_0$ depending only on  $\mu$, $\lambda$, $\mu_r$, $c_0$, $c_a$, $c_d$, $j_I$, $c_v$, $R$, $\kappa$, and $\|\rho_0\|_{L^1}$,
 such that if
\begin{align}\label{1.8}
N_0\triangleq\bar{\rho}\big[\|\rho_0\|_{L^3}+\bar{\rho}^2
\big(\|\sqrt{\rho_0}u_0\|_{L^2}^2+\|\sqrt{\rho_0}w_0\|_{L^2}^2\big)\big]
\big(\|\nabla u_0\|_{L^2}^2+\|w_0\|_{H^1}^2+
\bar{\rho}\|\sqrt{\rho_0}\Psi_0\|_{L^2}^2\big)\le\varepsilon_0,
\end{align}
where $\Psi_0\triangleq\frac{|u_0|^2}{2}+\frac{j_I|w_0|^2}{2}+c_v\theta_0$,
the problem \eqref{a1}--\eqref{a3} has a unique global strong solution $(\rho, u, w, \theta)$ satisfying, for any $T>0$,
\begin{align}\label{trr}
\begin{cases}
\rho \in C([0, T]; L^1\cap H^1\cap W^{1, q}), \ \rho_t\in C([0, T]; L^2\cap L^q),\\
(u, \theta)\in C([0, T]; D_0^1\cap D^2)\cap L^2([0, T]; D^{2, q}),\\
w\in C([0, T]; H^2)\cap L^2([0, T]; D^{2, q}),\\
(\nabla u_t, \nabla w_t, \nabla \theta_t)\in L^2([0, T]; L^2), \\
 (\sqrt{\rho}u_t, \sqrt{\rho}w_t, \sqrt{\rho}\theta_t)\in L^\infty([0, T]; L^2).
\end{cases}
\end{align}
\end{theorem}

\begin{remark}
Our Theorem \ref{thm1} generalizes the main result in \cite{LZ22} in the sense that the initial mass could be large.
\end{remark}

\begin{remark}
It should be noted that our smallness condition \eqref{1.8} is independent of any norms of the initial data except $\|\rho_0\|_{L^1}$, which is in sharp contrast to \cite{HLZ21} where they established global strong solution under small initial energy depending on the initial velocity, micro-rotation velocity, and temperature in some sense.
\end{remark}

\begin{remark}
The conclusion in Theorem \ref{thm1} generalizes global strong solutions showed in \cite{LZ20,WJ21} to the case of large oscillations since the smallness of the mean-square norm of $\big(\sqrt{\rho_0}u_0,\sqrt{\rho_0}w_0,\sqrt{\rho_0}\theta_0\big)$ implies \eqref{1.8}. In addition, both the initial density and the initial temperature are allowed to vanish at infinity.
\end{remark}

\begin{remark}
Very recently, Li and Zheng \cite{LJZ22} investigated well-posedness theory of compressible heat-conducting Navier-Stokes equations without initial compatibility condition. That is, the compatibility condition \eqref{1.7} may not be inevitable or can be released. We will put this question for future studies.
\end{remark}

\begin{remark}
The condition \eqref{1.5} is only used in obtaining estimates in Lemma \ref{l32}, and it could be relaxed by delicate analysis.
\end{remark}

We now make some comments on the analysis for Theorem \ref{thm1}.
The key issue in this paper is to derive both the
time-independent upper bound for the density and the time-dependent higher norm estimates of the solution $(\rho,u,w,\theta)$. First, due to the structure of \eqref{a1}, the basic energy estimate does not provide any dissipation estimate of the velocity and micro-rotation velocity. To overcome this difficulty, with the help of the entropy inequality and the conservation of mass, the authors \cite{HLZ21,LZ22} recovered some useful dissipation estimates in the cases of non-vacuum and vacuum at infinity, respectively. It should be noticed that the entropy inequality (an important tool in \cite{HLZ21}) only holds for a non-vacuum far field and the finiteness of mass is crucial in \cite{LZ22}. Motivated by \cite{JL20}, we obtain an useful basic inequality in terms of $\int\|\rho\|_{L^3}^2\|\nabla\theta\|_{L^2}^2dt$ (see Lemma \ref{l31}). Next, we attempt to deduce the estimates on the $L^\infty(0,T;L^2)$-norm of $\nabla u$ and $\nabla w$, which relies on
bounds of $\int_0^T\|(|u||\nabla u|,|u||\nabla w|,|w||\nabla u|,|w||\nabla w|)\|_{L^2}^2dt$ (see Lemma \ref{l34}).
Based on the temperature equation in the conservative form for the total energy $\rho\Psi=\rho\big(\frac{|u|^2}{2}+\frac{j_I|w|^2}{2}+c_v\theta\big)$,
we succeed in deriving the desired estimate from the momentum and angular momentum equations by choosing suitable multipliers (see Lemma \ref{l32}). Having all these estimates at hand, we can get time-independent estimates of the density and the quantity $N_T$ (its expression is given in Lemma \ref{l36}) provided that $N_0$ is suitably small (see Lemma \ref{l37}). Next, the main step is to bound the gradient of temperature.
By the basic estimates of the material derivatives of both the velocity and the micro-rotation velocity, we obtain $L^\infty(0,T;L^2)$-norm of $\nabla \theta$ (see Lemma \ref{l41}). Then, one can show higher order estimates of the temperature by careful analysis on the material derivative (see Lemma \ref{l42}). Finally, based on a Beale-Kato-Majda type inequality (see Lemma \ref{l24}) and the \textit{a priori} estimates we have just derived, we deduce the bounds on the $L^\infty(0,T;L^2\cap L^q)$-norm of $\nabla \rho$ and higher order estimates of the velocity and the micro-rotation velocity (see Lemma \ref{lz}).

The rest of the paper is organized as follows. In Section \ref{sec2}, we recall some known facts and elementary inequalities which will be used later. Sections \ref{sec3} and \ref{sec4} are devoted to obtaining the global {\it a priori} estimates. Finally, we give the proof of Theorem \ref{thm1} in Section \ref{sec5}.

\section{Preliminaries}\label{sec2}
In this section, we will recall some known facts and elementary inequalities which will be used later.

We begin with the local existence and uniqueness of strong solutions to the problem \eqref{a1}--\eqref{a3}, whose proof can be performed
by using similar strategies as those in \cite{CK120}.
\begin{lemma}\label{l21}
Assume that $(\rho_0, u_0, w_0,\theta_0)$ satisfies \eqref{1.6} and \eqref{1.7}, then there is a small $T_*>0$ such that the Cauchy problem \eqref{a1}--\eqref{a3} admits a unique solution $(\rho, u, w, \theta)$
in $\mathbb{R}^3\times[0, T_*]$. Furthermore, we have
\begin{align}\label{w1}
&\sup_{0\le t\le T_*}\big(\|\nabla u\|_{H^1}^2+\|\nabla\theta\|_{H^1}^2+\|w\|_{H^2}^2
+\|\sqrt{\rho}\dot{u}\|_{L^2}^2+\|\sqrt{\rho}\dot{\theta}\|_{L^2}^2
+\|\rho\dot{w}\|_{L^2}^2\big)\nonumber\\
&\quad +\int_0^{T_*}\big(\|\nabla\dot{u}\|_{L^2}^2+\|\nabla\dot{w}\|_{L^2}^2
+\|\nabla\dot{\theta}\|_{L^2}^2\big)dt\le M_*,
\end{align}
where the constant $M_*>1$ depends on $T_*$ and the initial data.
\end{lemma}

Next, the following well-known Gagliardo-Nirenberg inequality (see \cite[Chapter 10, Theorem 1.1]{D2016}) will be
used later.
\begin{lemma}\label{l22}
Assume that $f\in D^{1, m}\cap L^r$ with $m,r\geq1$, then there exists a constant $C$ depending only on $q$, $m$, and $r$ such that
\begin{align*}
\|f\|_{L^q}\le C\|\nabla f\|_{L^m}^\vartheta\|f\|_{L^r}^{1-\vartheta},
\end{align*}
where $\vartheta=\big(\frac{1}{r}-\frac{1}{q}\big)/
\big(\frac{1}{r}-\frac{1}{m}+\frac13\big)$ and the admissible range of $q$ is the following:
\begin{itemize}
\item if $m<3$, then $q$ is between $r$ and $\frac{3m}{3-m}$;
\item if $m=3$, then $q\in [r, \infty)$;
\item if $m>3$, then $q\in [r, \infty]$.
\end{itemize}
\end{lemma}

Next, from the system \eqref{a1}, we have
\begin{align}
\begin{cases}
\rho\dot{u}=\nabla F_1-(\mu+\mu_r)\curl\curl u+2\mu_r\curl w,\\
j_I\rho\dot{w}+4\mu_rw=\nabla F_2-(c_a+c_d)\curl\curl w
+2\mu_r\curl u,
\end{cases}
\end{align}
with
\begin{align}\label{2.3}
F_1\triangleq(2\mu+\lambda)\divv u-p,\ F_2\triangleq(2c_d+c_0)\divv w.
\end{align}
Furthermore, we have the following elliptic system
\begin{align}\label{2.5}
\begin{cases}
\Delta F_1=\divv(\rho\dot{u}),\\
\Delta F_2=j_I{\rm div}(\rho\dot{w})+4\mu_r\divv w,\\
(\mu+\mu_r)\Delta(\curl u)=\curl(\rho\dot{u})-2\mu_r\curl\curl w,\\
(c_a+c_d)\Delta(\curl w)-4\mu_r\curl w=j_I\curl(\rho\dot{w})
-2\mu_r\curl\curl u.
\end{cases}
\end{align}
From the standard $L^r$-estimate of the system \eqref{2.5},
we have the following estimates on $(\rho, u, w, \theta)$,
which can be proved similarly by \cite{HLZ21}.
\begin{lemma}\label{l23}
Suppose that $(\rho, u, w, \theta)$ is a smooth solution to \eqref{a1}--\eqref{a3}, then there exists a positive $C$ such that
the following estimates hold
\begin{align}
\|\nabla u\|_{L^r}&\le C\big(\|F_1\|_{L^r}+\|\rho\theta\|_{L^r}
+\|\curl u\|_{L^r}\big),\\
\|\nabla w\|_{L^r}&\le C\big(\|F_2\|_{L^r}+\|\curl w\|_{L^r}\big),\\
\|\nabla F_1\|_{L^r}+\|\nabla\curl u\|_{L^r}&\le C\big(\|\rho\dot{u}\|_{L^r}+\|\nabla w\|_{L^r}\big),\\
\|\nabla F_2\|_{L^r}+\|\nabla\curl w\|_{L^r}&\le C\big(\|\rho\dot{w}\|_{L^r}+\|\nabla u\|_{L^r}+\|w\|_{L^r}\big),\\
\|\nabla u\|_{L^6}&\le C\big(\|\rho\dot{u}\|_{L^2}+\|\nabla w\|_{L^2}+\|\rho\theta\|_{L^6}\big),\\
\|\nabla w\|_{L^6}&\le C\big(\|\rho\dot{w}\|_{L^2}+\|\nabla u\|_{L^2}+\|w\|_{L^2}\big).
\end{align}
where $2\leq r\leq6$.
\end{lemma}

Finally, the following Beale-Kato-Majda type inequality (see \cite[Lemma 2.3]{HLX20112}) will be used to estimate $\|\nabla u\|_{L^\infty}$.
\begin{lemma}\label{l24}
Suppose that $\nabla u\in L^2\cap W^{1, r}$ for some $r\in (3, \infty)$, then there exists a constant $C$ depending only on $r$ such that
\begin{align}
\|\nabla u\|_{L^\infty}\le C\big(\|\divv u\|_{L^\infty}+\|\curl u\|_{L^\infty}\big)\ln\big(e+\|\nabla^2u\|_{L^r}\big)
+C\big(1+\|\nabla u\|_{L^2}\big).
\end{align}
\end{lemma}

\section{Time independent lower order \textit{a priori} estimates}\label{sec3}
In this and following sections, we will establish some necessary {\it a priori} bounds for
smooth solutions to the Cauchy problem \eqref{a1}--\eqref{a3} to extend the local strong
solutions guaranteed by Lemma \ref{l21}. Thus, let $T>0$ be a fixed time and
$(\rho, u, w, \theta)$ be the smooth solution to \eqref{a1}--\eqref{a3} in $\mathbb{R}^3\times(0, T]$ with initial data
$(\rho_0, u_0, w_0, \theta_0)$ satisfying \eqref{1.6} and \eqref{1.7}. In what follows, $C$ denotes a generic constant which relies only on  $\mu$, $\lambda$, $\mu_r$, $c_0$, $c_a$, $c_d$, $j_I$, $c_v$, $R$, $\kappa$, and $\|\rho_0\|_{L^1}$; $1<C_i<C_{i+1}\ (i=1, 2, \cdots)$ denote different constants. Sometimes we use $C(\alpha)$ to emphasize the dependence of $C$ on $\alpha$. Moreover, for simplicity, we write
\begin{align*}
\int\cdot dx=\int_{\mathbb{R}^3}\cdot dx.
\end{align*}

Firstly, if we multiply $\eqref{a1}_1$ by a cut-off function and then use a standard limit procedure, we can derive that (see for example \cite[Lemma 3.1]{WZ17})
\begin{align}\label{3.1}
\|\rho(t)\|_{L^1}=\|\rho_0\|_{L^1}, \quad 0\leq t\leq T.
\end{align}
Applying the maximum principle (see \cite[p. 43]{F2004}) to \eqref{a1}$_4$ along with $\theta_0\geq0$ shows that
\begin{equation*}
\theta(x,t)\geq0\ \ \text{for any}\ \ (x,t)\in \mathbb{R}^3\times[0,T].
\end{equation*}
Moreover, the non-negativity of the density $\rho(x,t)$ follows from \eqref{a1}$_1$ and $\rho_0\geq0$ (see also \cite[p. 43]{F2004}).

We begin with the following standard energy estimate.
\begin{lemma}\label{l31}
There exists a positive constant $C$ depending only on $R$, $\mu$, $\lambda$, $\mu_r$, $c_0$, $c_a$, and $c_d$ such that
\begin{align}\label{3.3}
&\sup_{0\le t\le T}\big(\|\sqrt{\rho}u\|_{L^2}^2+\|\sqrt{\rho}w\|_{L^2}^2\big)+
\int_0^T\big(\|\nabla u\|_{L^2}^2+\|w\|_{H^1}^2\big)dt\nonumber\\
&\le \|\sqrt{\rho_0}u_0\|_{L^2}^2+\|\sqrt{\rho_0}w_0\|_{L^2}^2
+C\int_0^T\|\rho\|_{L^3}^2\|\nabla\theta\|_{L^2}^2dt.
\end{align}
\end{lemma}
\begin{proof}[Proof]
Multiplying $\eqref{a1}_2$ by $u$ and $\eqref{a1}_3$ by $w$, respectively, and integrating by parts over $\mathbb{R}^3$,
we then obtain that
\begin{align}\label{z3.3}
&\frac{1}{2}\frac{d}{dt}\int\big(\rho|u|^2+\rho|w|^2\big)dx
+\int\big[(\mu+\mu_r)|\nabla u|^2+(\lambda+\mu-\mu_r)
(\divv u)^2\big]dx\nonumber\\
&\quad+\int\big[(c_a+c_d)|\nabla w|^2+(c_0+c_d-c_a)
(\divv w)^2+4\mu_r|w|^2\big]dx\nonumber\\
&=\frac{1}{2}\frac{d}{dt}\int\big(\rho|u|^2+\rho|w|^2\big)dx
+\int\big[\mu|\nabla u|^2+(\lambda+\mu)(\divv u)^2
+\mu_r|\curl u|^2\big]dx\nonumber\\
&\quad+\int\big[c_d|\nabla w|^2+(c_0+c_d)(\divv w)^2
+c_a|\curl w|^2+4\mu_r|w|^2\big]dx\nonumber\\
&=R\int\rho\theta\divv udx+4\mu_r\int\curl u\cdot wdx\nonumber\\
&\le R\|\rho\|_{L^3}\|\theta\|_{L^6}\|\divv u\|_{L^2}
+4\mu_r\int\curl u\cdot wdx\nonumber\\
&\le C\|\rho\|_{L^3}\|\nabla\theta\|_{L^2}\|\divv u\|_{L^2}
+4\mu_r\int\curl u\cdot wdx\nonumber\\
&\le \frac{\mu+\lambda}{2}\|\divv u\|_{L^2}^2+C\|\rho\|_{L^3}^2\|\nabla\theta\|_{L^2}^2
+4\mu_r\int\curl u\cdot wdx,
\end{align}
where in the first equality one has used
\begin{align}\label{z3.4}
\int|\nabla u|^2dx=\int\big(|\curl u|^2+(\divv u)^2\big)dx,\ \
\int|\nabla w|^2dx=\int\big(|\curl w|^2+(\divv w)^2\big)dx.
\end{align}
Thus, we derive from \eqref{z3.3} that
\begin{align}\label{3.4}
&\frac{d}{dt}(\|\sqrt{\rho}u\|_{L^2}^2+\|\sqrt{\rho}w\|_{L^2}^2)+\mu\|\nabla u\|_{L^2}^2
+(\mu+\lambda)\|\divv u\|_{L^2}^2+c_d\|\nabla w\|_{L^2}^2\nonumber\\
&+\mu_r\|\curl u-2w\|_{L^2}^2+(c_0+c_d)\|\divv w\|_{L^2}^2
+c_a\|\curl w\|_{L^2}^2\nonumber\\
&\le C\|\rho\|_{L^3}^2\|\nabla\theta\|_{L^2}^2,
\end{align}
due to
\begin{align*}
\int\big(\mu_r|\curl u|^2-4\mu_r\curl u\cdot w+4\mu_r|w|^2\big)dx=\mu_r\int|\curl u-2w|^2dx.
\end{align*}
Integrating \eqref{3.4} in $t$ over $[0,T]$ leads to \eqref{3.3}.
\end{proof}

\begin{lemma}\label{l32}
Under the condition \eqref{1.5}, there exists a positive constant $C$ depending only on $\mu$, $\lambda$, $\mu_r$, $c_0$, $c_a$, $c_d$, $j_I$, $c_v$, $R$, and $\kappa$ such that, for any $\varepsilon>0$,
\begin{align}\label{3.5}
&\sup_{0\le t\le T}\|\sqrt{\rho}\Psi\|_{L^2}^2
+\int_0^T\big(\|\nabla\theta\|_{L^2}^2+\||u||\nabla u|\|_{L^2}^2+\||w||\nabla w|\|_{L^2}^2\big)dt\nonumber\\
&\quad+\int_0^T\big(\||w||\nabla u|\|_{L^2}^2
+\||u||\nabla w|\|_{L^2}^2+\||u||w|\|_{L^2}^2\big)dt\nonumber\\
&\le C\|\sqrt{\rho_0}\Psi_0\|_{L^2}^2
+C\int_0^T\|\rho\|_{L^\infty}\|\sqrt{\rho}\theta\|_{L^2}\|\nabla\theta\|_{L^2}\||u||\nabla u|\|_{L^2}\|\rho\|_{L^3}^\frac{1}{2}dt\nonumber\\
&\quad +C(\varepsilon)\bar{\rho}^2\int_0^T\|\nabla u\|_{L^2}^4\|\nabla w\|_{L^2}^2dt+\frac{C\varepsilon}{\bar{\rho}^2}\int_0^T\|w\|_{L^2}^2dt
+C\bar{\rho}^2\int_0^T\|\nabla w\|_{L^2}^4\|\nabla u\|_{L^2}^2dt,
\end{align}
where $\Psi\triangleq\frac{|u|^2}{2}+\frac{j_I|w|^2}{2}+c_v\theta$.
\end{lemma}
\begin{proof}[Proof]
1. For $\Psi\triangleq\frac{|u|^2}{2}+\frac{j_I|w|^2}{2}+c_v\theta$,
we derive from \eqref{a1} that
\begin{align}\label{3.13}
\rho (\Psi_t+u\cdot\nabla\Psi)+\divv(up)-\kappa\Delta\theta
& =\divv(\mathcal{S}\cdot u)+\divv(\mathcal{G}\cdot w)
+\curl(\mathcal{R}\cdot u) \notag \\
& \quad -2\mu_r\curl u\cdot w+2\mu_r\curl w\cdot u,
\end{align}
where $\mathcal{S}=(\mu+2\mu_r)(\nabla u+\nabla u^\top)+(\lambda-4\mu_r)\divv u
\mathbb{I}_3$, $\mathcal{G}=(c_d+c_a)(\nabla w+\nabla w^\top)+(c_0-2c_a)\divv w\mathbb{I}_3$, and $\mathcal{R}=\mu_r\curl u$ with $\mathbb{I}_3$ being the identity matrix of order $3$. Multiplying \eqref{3.13} by $\Psi$ and integrating the resultant over $\mathbb{R}^3$, we deduce from Sobolev's inequality and Young's inequality that
\begin{align*}
&\frac12\frac{d}{dt}\|\sqrt{\rho}\Psi\|_{L^2}^2
+c_v\kappa\|\nabla\theta\|_{L^2}^2\nonumber\\
&\le \frac{c_v\kappa}{2}\|\nabla\theta\|_{L^2}^2+C\||u||\nabla u|\|_{L^2}^2
+C\||w||\nabla w|\|_{L^2}^2+C\int\rho^2\theta^2|u|^2dx\nonumber\\
&\quad+C\int|\nabla u||w||\Psi|dx+C\int|w||u||\nabla\Psi|dx\nonumber\\
&\le\frac{c_v\kappa}{4}\|\nabla\theta\|_{L^2}^2+C\||u||\nabla u|\|_{L^2}^2
+C\||w||\nabla w|\|_{L^2}^2+C\|\nabla u\|_{L^2}\|w\|_{L^3}\|\Psi\|_{L^6}\nonumber\\
&\quad+C\|u\|_{L^6}\|w\|_{L^3}\|\nabla\Psi\|_{L^2}
+C\int\rho^2\theta^2|u|^2dx\nonumber\\
&\le \frac{c_v\kappa}{4}\|\nabla\theta\|_{L^2}^2+C\||u||\nabla u|\|_{L^2}^2
+C\||w||\nabla w|\|_{L^2}^2+C\int\rho^2\theta^2|u|^2dx\nonumber\\
&\quad+C\|\nabla\Psi\|_{L^2}\|\nabla u\|_{L^2}\|w\|_{L^2}^\frac12\|\nabla w\|_{L^2}^\frac12\nonumber\\
&\le \frac{c_v\kappa}{2}\|\nabla\theta\|_{L^2}^2+\frac{\varepsilon}{\bar{\rho}^2}\|w\|_{L^2}^2+C\||u||\nabla u|\|_{L^2}^2
+C\||w||\nabla w|\|_{L^2}^2+C\int\rho^2\theta^2|u|^2dx\nonumber\\
&\quad+C(\varepsilon)\bar{\rho}^2\|\nabla u\|_{L^2}^4\|\nabla w\|_{L^2}^2,
\end{align*}
which yields that
\begin{align}\label{3.14}
\frac12\frac{d}{dt}\|\sqrt{\rho}\Psi\|_{L^2}^2
+\frac{c_v\kappa}{2}\|\nabla\theta\|_{L^2}^2
& \le \frac{\varepsilon}{\bar{\rho}^2}\|w\|_{L^2}^2+C\||u||\nabla u|\|_{L^2}^2
+C\||w||\nabla w|\|_{L^2}^2  \nonumber\\
& \quad +C\int\rho^2\theta^2|u|^2dx
+C(\varepsilon)\bar{\rho}^2\|\nabla u\|_{L^2}^4\|\nabla w\|_{L^2}^2.
\end{align}

2. Multiplying $\eqref{a1}_2$ by $|u|^2u$ and integration by parts, we obtain that
\begin{align}\label{3.6}
&\frac14\frac{d}{dt}\|\sqrt{\rho}|u|^2\|_{L^2}^2-\int\big[(\mu+\mu_r)\Delta u+(\mu+\lambda-\mu_r)\nabla\divv u\big]\cdot|u|^2udx\nonumber\\
&=R\int \rho\theta\divv(|u|^2u)dx+2\mu_r\int \curl w\cdot|u|^2udx\nonumber\\
&\le \Big(\mu-\frac{\lambda}{2}-\frac{\mu_r}{4}\Big)\int|u|^2|\nabla u|^2dx+C\int\rho^2\theta^2|u|^2dx+2\mu_r\int w\cdot\curl(|u|^2u)dx.
\end{align}
Using the following facts
\begin{align*}
|\curl u|\le |\nabla u|,\
|\nabla|u||\le |\nabla u|,\ \Delta u=\nabla\divv u-\curl\curl u,\ \mu_r>0,\
2\mu>\lambda+\frac{\mu_r}{2},
\end{align*}
we deduce that
\begin{align}\label{3.7}
&-\int\big[(\mu+\mu_r)\Delta u+(\mu+\lambda-\mu_r)\nabla\divv u\big]\cdot|u|^2udx\nonumber\\
&=-\int|u|^2(\mu\Delta u+(\mu+\lambda)\nabla \divv u\cdot u)dx+\mu_r\int|u|^2\curl\curl u\cdot udx\nonumber\\
&=\int|u|^2\big[\mu|\nabla u|^2+(\mu+\lambda)
(\divv u)^2+2\mu|\nabla|u||^2\big]dx+(\mu+\lambda)\int(\nabla|u|^2)\cdot
u\divv udx\nonumber\\
&\quad+\mu_r\int|\curl u|^2|u|^2dx+\mu_r\int\curl u\cdot\nabla|u|^2\times udx\nonumber\\
&\ge \int|u|^2\big[\mu|\nabla u|^2+(\mu+\lambda)
(\divv u)^2+2\mu|\nabla|u||^2-2(\mu+\lambda)|\nabla|u|||\divv u|\big]dx\nonumber\\
&\quad+\frac{\mu_r}{2}\int|\curl u|^2|u|^2dx-\frac{\mu_r}{2}\int|\nabla|u||^2|u|^2dx\nonumber\\
&=\int|u|^2\big[\mu|\nabla u|^2+(\mu+\lambda)(\divv u-|\nabla|u||)^2\big]dx
+\frac{\mu_r}{2}\int|\curl u|^2|u|^2dx\nonumber\\
&\quad+\Big(\mu-\lambda-\frac{\mu_r}{2}\Big)\int|u|^2|\nabla|u||^2dx\nonumber\\
&\ge 2\Big(\mu-\frac{\lambda}{2}-\frac{\mu_r}{4}\Big)\int|u|^2|\nabla u|^2dx,
\end{align}
which combined with \eqref{3.6} yields that
\begin{align}\label{3.8}
\frac14\frac{d}{dt}\|\sqrt{\rho}|u|^2\|_{L^2}^2
+\Big(\mu-\frac{\lambda}{2}-\frac{\mu_r}{4}\Big)\||u||\nabla u|\|_{L^2}^2
\le C\int\rho^2\theta^2|u|^2dx+C\int|w||u|^2|\nabla u|dx.
\end{align}

3. Multiplying $\eqref{a1}_3$ by $|w|^2w$ and integration by parts, we derive that
\begin{align}\label{3.9}
&\frac{j_I}{4}\frac{d}{dt}\|\sqrt{\rho}|w|^2\|_{L^2}^2-\int\big[(c_a+c_d)\Delta w+(c_0+c_d-c_a)\nabla\divv w\big]\cdot|w|^2wdx+4\mu_r\int|w|^4dx \notag \\
&=2\mu_r\int u\cdot\curl(|w|^2w)dx.
\end{align}
Owing to $c_a>0$ and $2c_d>c_0+\frac{c_a}{2}$, similarly to \eqref{3.7}, we deduce that
\begin{align*}
-\int\big[(c_d+c_0)\Delta u+(c_d+c_0-c_a)\nabla\divv w\big]\cdot|w|^2wdx
\ge \Big(2c_d-c_0-\frac{c_a}{2}\Big)\int|w|^2|\nabla w|^2dx.
\end{align*}
This along with \eqref{3.9} gives that
\begin{align}\label{z3.9}
\frac14\frac{d}{dt}\|\sqrt{\rho}|w|^2\|_{L^2}^2
+\frac{1}{j_I}\Big(2c_d-c_0-\frac{c_a}{2}\Big)\||w||\nabla w|\|_{L^2}^2+\frac{4\mu_r}{j_I}\|w\|_{L^4}^4\leq C\int |u||w|^2|\nabla w|dx.
\end{align}
We then infer from \eqref{3.8}, \eqref{z3.9}, Sobolev's inequality, and Cauchy-Schwarz inequality that, for $\varepsilon>0$,
\begin{align*}
&\frac14\frac{d}{dt}\big(\|\sqrt{\rho}|u|^2\|_{L^2}^2
+\|\sqrt{\rho}|w|^2\|_{L^2}^2\big)
+\Big(\mu-\frac{\lambda}{2}-\frac{\mu_r}{4}\Big)\||u||\nabla u|\|_{L^2}^2 \notag \\
& \quad +\frac{1}{j_I}\Big(2c_d-c_0-\frac{c_a}{2}\Big)\||w||\nabla w|\|_{L^2}^2+\frac{4\mu_r}{j_I}\|w\|_{L^4}^4\nonumber\\
&\le C\int\rho^2\theta^2|u|^2dx+C\int|w||u|^2|\nabla u|dx+C\int|u||w|^2|\nabla w|dx\nonumber\\
&\le C\int\rho^2\theta^2|u|^2dx
+C\|w\|_{L^3}\||u|^2\|_{L^6}\|\nabla u\|_{L^2}+C\|u\|_{L^6}\||w||\nabla w|\|_{L^2}\|w\|_{L^3}\nonumber\\
&\le C\int\rho^2\theta^2|u|^2dx
+C\|w\|_{L^3}\||u||\nabla u|\|_{L^2}\|\nabla u\|_{L^2}
+C\|w\|_{L^3}\|\nabla u\|_{L^2}\||w||\nabla w|\|_{L^2}\nonumber\\
&\le \frac12\Big(\mu-\frac{\lambda}{2}-\frac{\mu_r}{4}\Big)\||u||\nabla u|\|_{L^2}^2
+\frac{1}{2j_I}\Big(2c_d-c_0-\frac{c_a}{2}\Big)\||w||\nabla w|\|_{L^2}^2 \notag \\
& \quad +C\|w\|_{L^2}\|\nabla w\|_{L^2}\|\nabla u\|_{L^2}^2+C\int\rho^2\theta^2|u|^2dx\nonumber\\
&\le \frac12\Big(\mu-\frac{\lambda}{2}-\frac{\mu_r}{4}\Big)\||u||\nabla u|\|_{L^2}^2
+\frac{1}{2j_I}\Big(2c_d-c_0-\frac{c_a}{2}\Big)\||w||\nabla w|\|_{L^2}^2 \notag \\
& \quad +C(\varepsilon)\bar{\rho}\|\nabla u\|_{L^2}^4\|\nabla w\|_{L^2}^2+\frac{\varepsilon}{\bar{\rho}}\|w\|_{L^2}^2
+C\int\rho^2\theta^2|u|^2dx,
\end{align*}
which implies that
\begin{align}\label{3.10}
&\frac{d}{dt}\big(\|\sqrt{\rho}|u|^2\|_{L^2}^2+\|\sqrt{\rho}|w|^2\|_{L^2}^2\big)
+2\Big(\mu-\frac{\lambda}{2}-\frac{\mu_r}{4}\Big)\||u||\nabla u|\|_{L^2}^2+\frac{2}{j_I}\Big(2c_d-c_0-\frac{c_a}{2}\Big)\||w||\nabla w|\|_{L^2}^2\nonumber\\
&\le C(\varepsilon)\bar{\rho}\|\nabla u\|_{L^2}^4\|\nabla w\|_{L^2}^2+\frac{C\varepsilon}{\bar{\rho}}\|w\|_{L^2}^2
+C\int\rho^2\theta^2|u|^2dx.
\end{align}

4. Multiplying $\eqref{a1}_2$ by $|w|^2u$ and $\eqref{a1}_3$ by $\frac{1}{j_I}|u|^2w$, respectively, summing up and integrating by parts over $\mathbb{R}^3$, one gets that
\begin{align*}
&\frac12\frac{d}{dt}\|\sqrt{\rho}|u||w|\|_{L^2}^2+\mu\||w||\nabla u|\|_{L^2}^2
+\mu_r\||w||\curl u|\|_{L^2}^2+(\mu+\lambda)\||w||\divv u|\|_{L^2}^2\nonumber\\
&\quad+\frac{c_d}{j_I}\||u||\nabla w|\|_{L^2}^2+\frac{c_0+c_d}{j_I}\||u||\divv w|\|_{L^2}^2+\frac{c_a}{j_I}\||u||\curl w|\|_{L^2}^2+\frac{4\mu_r}{j_I}\||u||w|\|_{L^2}^2\nonumber\\
&=-\frac{\mu j_I+c_d}{2j_I}\int\nabla(|u|^2)\cdot\nabla(|w|^2)dx
-(\mu+\lambda)\int u\cdot\nabla(|w|^2)\divv udx
+\int p\divv (|w|^2u)dx\nonumber\\
&\quad-\mu_r\int\curl u\cdot\nabla|w|^2\times udx-\frac{c_0+c_d}{j_I}\int w\cdot\nabla(|u|^2)\divv wdx
-\frac{c_a}{j_I}\int\curl w\cdot\nabla(|u|^2)\times wdx\nonumber\\
&\quad+2\mu_r\int\curl w\cdot|w|^2udx+2\mu_r\int\curl u\cdot|u|^2wdx\nonumber\\
&\le C\||u||\nabla u|\|_{L^2}^2+C\||w||\nabla w|\|_{L^2}^2+2\mu_r\||u||w|\|_{L^2}^2
+C\int\rho^2\theta^2|u|^2dx+C\bar{\rho}\int\theta|\divv u||w|^2dx\nonumber\\
&\le  C\||u||\nabla u|\|_{L^2}^2+C\||w||\nabla w|\|_{L^2}^2+2\mu_r\||u||w|\|_{L^2}^2
+C\int\rho^2\theta^2|u|^2dx+C\bar{\rho}\|\theta\|_{L^6}\|\nabla u\|_{L^2}\||w|^2\|_{L^3}\nonumber\\
&\le \frac{c_v\kappa}{4}\|\nabla\theta\|_{L^2}^2+C\||u||\nabla u|\|_{L^2}^2+C\||w||\nabla w|\|_{L^2}^2+2\mu_r\||u||w|\|_{L^2}^2
+C\bar{\rho}^2\|\nabla u\|_{L^2}^2\|\nabla w\|_{L^2}^4,
\end{align*}
which combined with \eqref{3.14} leads to
\begin{align}\label{3.12}
&\frac{d}{dt}\big(\|\sqrt{\rho}\Psi\|_{L^2}^2+\|\sqrt{\rho}|u||w|\|_{L^2}^2\big)
+\|\nabla\theta\|_{L^2}^2+\||w||\nabla u|\|_{L^2}^2
+\||u||\nabla w|\|_{L^2}^2+\||u||w|\|_{L^2}^2\nonumber\\
&\le C\||u||\nabla u|\|_{L^2}^2+C\||w||\nabla w|\|_{L^2}^2
+C\bar{\rho}^2\|\nabla u\|_{L^2}^2\|\nabla w\|_{L^2}^4.
\end{align}
Adding \eqref{3.12} to \eqref{3.10} multiplied by a large enough constant $K>0$, we deduce that
\begin{align}\label{3.15}
&\frac{d}{dt}\big(\|\sqrt{\rho}\Psi\|_{L^2}^2+\|\sqrt{\rho}|u||w|\|_{L^2}^2
+\|\sqrt{\rho}|u|^2\|_{L^2}^2+\|\sqrt{\rho}|w|^2\|_{L^2}^2\big)
\nonumber\\
&\quad+\|\nabla\theta\|_{L^2}^2
+\||u||\nabla u|\|_{L^2}^2+\||w||\nabla w|\|_{L^2}^2+\||w||\nabla u|\|_{L^2}^2+\||u||\nabla w|\|_{L^2}^2+\||u||w|\|_{L^2}^2\nonumber\\
&\le C\|\rho\|_{L^\infty}\|\sqrt{\rho}\theta\|_{L^2}\|\nabla\theta\|_{L^2}\||u||\nabla u|\|_{L^2}\|\rho\|_{L^3}^\frac{1}{2}
+C(\varepsilon)\bar{\rho}^2\|\nabla u\|_{L^2}^4\|\nabla w\|_{L^2}^2\nonumber\\
&\quad+C\bar{\rho}^2\|\nabla u\|_{L^2}^2\|\nabla w\|_{L^2}^4+\frac{C \varepsilon}{\bar{\rho}^2}\|w\|_{L^2}^2,
\end{align}
due to
\begin{align}\label{3.16}
C\int\rho^2\theta^2|u|^2dx\le C\|\sqrt{\rho}\theta\|_{L^2}\|\theta\|_{L^6}\||u|^2\|_{L^6}
\|\rho\|_{L^9}^\frac32\le C\|\rho\|_{L^\infty}\|\sqrt{\rho}\theta\|_{L^2}\|\nabla\theta\|_{L^2}\||u||\nabla u|\|_{L^2}\|\rho\|_{L^3}^\frac{1}{2}.
\end{align}
Thus, \eqref{3.5} follows from \eqref{3.15} by integrating $t$ over $[0,T]$.
\end{proof}

\begin{lemma}
Assume that
\begin{align}\label{3.17}
\sup_{0\le t\le T}\|\rho\|_{L^\infty}\le 4\bar{\rho},
\end{align}
then it holds that
\begin{align}\label{3.18}
\sup_{0\le t\le T}\|\rho\|_{L^3}^3+\int_0^T\int\rho^3pdxdt &\le \|\rho_0\|_{L^3}^3
+C\bar{\rho}^\frac23\sup_{0\le t\le T}\Big(\|\sqrt{\rho}u\|_{L^2}^\frac13
\|\sqrt{\rho}|u|^2\|_{L^2}^\frac13\|\rho\|_{L^3}^3\Big)\nonumber\\
&\quad+C\bar{\rho}^2\int_0^T\|\rho\|_{L^3}^2\|\nabla u\|_{L^2}^2dt.
 \end{align}
\end{lemma}
\begin{proof}[Proof]
Due to $\divv\curl w=0$, we can rewrite $\eqref{a1}_2$ as
\begin{align}\label{dxy}
\Delta^{-1}\divv(\rho u)_t+\Delta^{-1}\divv\divv(\rho u\otimes u)-(2\mu+\lambda)\divv u+p=0.
\end{align}
Thus, we obtain \eqref{3.18} by the same arguments as those in \cite[Proposition 2.4]{JL20}. Here we omit the details for simplicity.
\end{proof}

\begin{lemma}\label{l34}
Under the condition \eqref{3.17}, it holds that
\begin{align}\label{3.21}
&\sup_{0\le t\le T}\big(\|\nabla u\|_{L^2}^2+\|w\|_{H^1}^2\big)
+\int_0^T\Big\|\Big(\sqrt{\rho}u_t, \sqrt{\rho}w_t, \frac{\nabla F_1}{\sqrt{\bar{\rho}}},
 \frac{\nabla F_2}{\sqrt{\bar{\rho}}}, \frac{\nabla V_1}{\sqrt{\bar{\rho}}},
  \frac{\nabla V_2}{\sqrt{\bar{\rho}}}\Big)\Big\|_{L^2}^2dt\nonumber\\
&\le C\big(\|\nabla u_0\|_{L^2}^2+\|w_0\|_{H^1}^2\big)
+C\bar{\rho}^3\int_0^T\big(\|\nabla u\|_{L^2}^4+\|\nabla w\|_{L^2}^4\big)\big(\|\nabla u\|_{L^2}^2
+\bar{\rho}\|\sqrt{\rho}\theta\|_{L^2}^2\big)dt\nonumber\\
&\quad+C\int_0^T\big(\bar{\rho}^2\|\rho\|_{L^3}^\frac12\|\sqrt{\rho}\theta\|_{L^2}
+\bar{\rho}\big)\big(\|\nabla\theta\|_{L^2}^2+\||u||\nabla u|\|_{L^2}^2+\||u||\nabla w|\|_{L^2}^2\big)dt\nonumber\\
&\quad+C\bar{\rho}\sup_{0\le t\le T}\|\sqrt{\rho}\theta\|_{L^2}^2
+\frac{C}{\bar{\rho}}\int_0^T\big(\|w\|_{L^2}^2+\|\nabla u\|_{L^2}^2\big)dt,
\end{align}
where $F_1$ and $F_2$ are the same as those of in \eqref{2.3}, $V_1\triangleq\curl u$, and $V_2\triangleq\curl w$.
\end{lemma}
\begin{proof}[Proof]
1. Multiplying $\eqref{a1}_2$ by $u_t$ and $\eqref{a1}_3$ by $w_t$, respectively, and integrating the resulting equations over $\mathbb{R}^3$, we get that
\begin{align*}
&\frac12\frac{d}{dt}\int\big[(\mu+\lambda)(\divv u)^2+\mu|\nabla u|_{L^2}^2
+\mu_r|\curl u|^2
+(c_0+c_d)(\divv w)^2+c_d|\nabla w|^2+c_a|\curl w|^2\big]dx\nonumber\\
&\quad+\int\big(\rho|u_t|^2+j_I\rho|w_t|^2\big)dx\nonumber\\
&=\int p\divv u_tdx-\int\rho u\cdot\nabla u\cdot u_tdx-j_I\int\rho u\cdot\nabla w\cdot w_tdx\nonumber\\
&\quad+\int\big(2\mu_r\curl w\cdot u_t+2u_r\curl u\cdot w_t-4\mu_r w\cdot w_t\big)dx,
\end{align*}
which yields that
\begin{align}\label{3.23}
&\frac12\frac{d}{dt}\big[(\mu+\lambda)\|\divv u\|_{L^2}^2+\mu\|\nabla u\|_{L^2}^2+(c_0+c_d)\|\divv w\|_{L^2}^2
+c_d\|\nabla w\|_{L^2}^2+c_a\|\curl w\|_{L^2}^2+\mu_r\|\curl u-2w\|_{L^2}^2\big]\nonumber\\
&\quad+\|\sqrt{\rho}u_t\|_{L^2}^2+j_I\|\sqrt{\rho}w_t\|_{L^2}^2\nonumber\\
&=\int p\divv u_tdx-\int\rho u\cdot\nabla u\cdot u_tdx-j_I\int\rho u\cdot \nabla w\cdot w_tdx.
\end{align}
By the definition of effective viscous flux $F_1$ (see \eqref{2.3}), we use
$\divv u=\frac{F_1+p}{2\mu+\lambda}$ to obtain that
\begin{align}
\int p\divv u_tdx&=\frac{d}{dt}\int p\divv udx
-\int p_t\divv udx\nonumber\\
&=\frac{d}{dt}\int \frac{p (F_1+p)}{2\mu+\lambda}dx-\frac{1}{2(2\mu+\lambda)}\frac{d}{dt}\|p\|_{L^2}^2
-\frac{1}{2\mu+\lambda}\int p_tF_1dx\nonumber\\
&=\frac{1}{2(2\mu+\lambda)}\frac{d}{dt}\|p\|_{L^2}^2
+\frac{1}{2\mu+\lambda}\frac{d}{dt}\int pF_1dx
-\frac{1}{2\mu+\lambda}\int p_tF_1dx.
\end{align}
It follows $\eqref{a1}_1$, $\eqref{a1}_4$, and the state equation that
\begin{align}
p_t&=-\divv (pu)-(\gamma-1)\big(p\divv u-\kappa\Delta\theta-\mathcal{Q}(\nabla u)\big)
+4(\gamma-1)\mu_r\Big|\frac12\curl u-w\Big|^2\nonumber\\
&\quad+(\gamma-1)[c_0(\divv w)^2+(c_a+c_d)\nabla w:\nabla w^\top+(c_d-c_a)\nabla w:\nabla w],
\end{align}
with $\gamma-1=\frac{R}{c_v}$, which leads to
\begin{align}\label{3.26}
\int p_tF_1dx&=\int\Big[(\gamma-1)(\mathcal{Q}(\nabla u)-p\divv u)F_1+4(
\gamma-1)\mu_r
\Big|\frac12\curl u-w\Big|^2F_1\Big]dx\nonumber\\
&\quad+\int(\gamma-1)[c_0(\divv w)^2+(c_a+c_d)\nabla w:\nabla w^\top+(c_d-c_a)\nabla w:\nabla w]F_1dx\nonumber\\
&\quad+\int(up-(\gamma-1)\kappa\nabla\theta)\cdot\nabla F_1dx.
\end{align}
In view of \eqref{z3.4} and \eqref{3.23}--\eqref{3.26}, we arrive at
\begin{align}\label{3.27}
&\frac12\frac{d}{dt}\Big(\mu\|V_1\|_{L^2}^2
+(c_a+c_d)\|V_2\|_{L^2}^2+\frac{\|F_1\|_{L^2}^2}{2\mu+\lambda}
+\frac{\|F_2\|_{L^2}^2}{2c_d+c_0}
+\mu_r\|V_1-2w\|_{L^2}^2\Big)+\|\sqrt{\rho}u_t\|_{L^2}^2
+\|\sqrt{\rho}w_t\|_{L^2}^2\nonumber\\
&=-\int\rho u\cdot\nabla u\cdot u_tdx-j_I\int\rho u\cdot \nabla w\cdot w_tdx
+\frac{1}{2\mu+\lambda}\int\big((\gamma-1)\kappa\nabla\theta-up\big)\cdot\nabla F_1dx\nonumber\\
&\quad-\frac{\gamma-1}{2\mu+\lambda}\big[c_0(\divv w)^2+(c_a+c_d)\nabla w:\nabla w^\top+(c_d-c_a)\nabla w:\nabla w\big]F_1dx\nonumber\\
&\quad-\frac{\gamma-1}{2\mu+\lambda}\int\Big[\mathcal{Q}(\nabla u)-p\divv u+4\mu_r
\Big|\frac12\curl u-w\Big|^2\Big]F_1dx,
\end{align}
where $F_2=(2c_d+c_0)\divv w$, $V_1=\curl u$, and $V_2=\curl w$.

2. Noting that
\begin{align*}
\Delta u=\nabla\divv u-\curl\curl u,\  \Delta w=\nabla\divv w-\curl\curl w,
\end{align*}
we rewrite $\eqref{a1}_2$
and $\eqref{a1}_3$ as follows
\begin{align}
&\rho u_t+\rho u\cdot\nabla u=\nabla F_1-(\mu+\mu_r)\curl V_1+2\mu_r\curl w,\label{3.28}\\[3pt]
&j_I(\rho w_t+\rho u\cdot\nabla w)+4\mu_rw=\nabla F_2-(c_a+c_d)\curl V_2+2\mu_r\curl u,\label{3.29}
\end{align}
Multiplying \eqref{3.28} by $\nabla F_1$ and \eqref{3.29} by $\nabla F_2$, respectively, we obtain after integration by parts that
\begin{align*}
\|\nabla F_1\|_{L^2}^2+\|\nabla F_2\|_{L^2}^2
&=\int\rho(u_t+u\cdot\nabla u)\cdot\nabla F_1dx+\int\big[j_I\rho(w_t+u\cdot\nabla w)\cdot\nabla F_2+4\mu_rw\cdot\nabla F_2\big]dx\nonumber\\
&\le \int\Big(\frac{|\nabla F_1|^2+|\nabla F_2|^2}{2}+2\bar{\rho}\rho|u_t|^2+2j_I^2\bar{\rho}\rho|w_t|^2\Big)dx
+\int\rho u\cdot\nabla u\cdot\nabla F_1dx\nonumber\\
&\quad+\int\big(j_I\rho u\cdot\nabla w\cdot\nabla F_2
+4\mu_rw\cdot\nabla F_2\big)dx,
\end{align*}
due to
\begin{align*}
\int\nabla F_1\cdot\curl V_1dx=0,
\int\nabla F_2\cdot\curl V_2dx=0, \int\nabla F_1\cdot\curl wdx=0,
\int\nabla F_2\cdot\curl udx=0,
\|\rho\|_{L^\infty}\le 4\bar{\rho}.
\end{align*}
Thus, we get that
\begin{align}\label{3.30}
\frac{\|\nabla F_1\|_{L^2}^2+\|\nabla F_2\|_{L^2}^2}{16\bar{\rho}}&\le
\frac{1}{4}\big(\|\sqrt{\rho}u_t\|_{L^2}^2+\|\sqrt{\rho}w_t\|_{L^2}^2\big)
+\frac{1}{8\bar{\rho}}\int\rho u\cdot\nabla u\cdot\nabla F_1dx\nonumber\\
&\quad+\frac{1}{8\bar{\rho}}\int\rho u\cdot\nabla w\cdot\nabla F_2dx
+\frac{\mu_r}{2\bar{\rho}}\int w\cdot\nabla F_2dx.
\end{align}
Similarly, one has
\begin{align}\label{3.31}
\frac{(\mu+\mu_r)^2\|\nabla V_1\|_{L^2}^2+(c_a+c_d)^2\|\nabla V_2\|_{L^2}^2}{16\bar{\rho}}
&\le
\frac{1}{4}\big(\|\sqrt{\rho}u_t\|_{L^2}^2+\|\sqrt{\rho}w_t\|_{L^2}^2\big)
+\frac{1}{8\bar{\rho}}\int\rho u\cdot\nabla u\cdot\curl V_1dx\nonumber\\
&\quad+\frac{1}{8\bar{\rho}}\int\rho u\cdot\nabla w\cdot\curl V_2dx
+\frac{\mu_r}{2\bar{\rho}}\int w\cdot\curl V_2dx\nonumber\\
&\quad+\frac{\mu_r}{2\bar{\rho}}\int V_1\cdot\curl V_2dx.
\end{align}
Substituting \eqref{3.30} and \eqref{3.31} into \eqref{3.27}, we find that
\begin{align}\label{z3.31}
&\frac12\frac{d}{dt}\Big(\mu\|V_1\|_{L^2}^2
+(c_a+c_d)\|V_2\|_{L^2}^2+\frac{\|F_1\|_{L^2}^2}{2\mu+\lambda}+\frac{\|F_2\|_{L^2}^2}{2c_d+c_0}
+\mu_r\|V_1-2w\|_{L^2}^2\Big)+\frac12\|\sqrt{\rho}u_t\|_{L^2}^2\nonumber\\
&\quad+\frac{j_I}{2}\|\sqrt{\rho}w_t\|_{L^2}^2+\frac{1}{16\bar{\rho}}
\big(\|\nabla F_1\|_{L^2}^2+\|\nabla F_2\|_{L^2}^2
+(\mu+\mu_r)^2\|\nabla V_1\|_{L^2}^2+(c_a+c_d)^2\|\nabla V_2\|_{L^2}^2\big)\nonumber\\
&\le C\int\rho|u||\nabla u|\Big(|u_t|+\frac{1}{\bar{\rho}}(|\nabla F_1|+|\nabla V_1|)\Big)dx+C\int\rho|u||\nabla w|
\Big(|w_t|+\frac{1}{\bar{\rho}}(|\nabla F_2|+|\nabla V_2|)\Big)dx\nonumber\\
&\quad+C\int(|\nabla\theta|+\rho\theta|u|)|\nabla F_1|dx
+C\int(|\nabla u|^2+|\nabla w|^2+\rho\theta|\nabla u|+|w|^2)|F_1|dx\nonumber\\
&\quad+\frac{C}{\bar{\rho}}\int|w||\nabla F_2|dx+\frac{C}{\bar{\rho}}\int(|w|+|\nabla u|)|\nabla V_2|dx
\triangleq\sum_{i=1}^6J_i.
\end{align}

3. It follows from H\"older's, Young's, and Gagliardo-Nirenberg inequalities that
\begin{align*}
J_1&\le C\bar{\rho}^\frac12\||u||\nabla u|\|_{L^2}\|\sqrt{\rho}u_t\|_{L^2}
+C\||u||\nabla u|\|_{L^2}(\|\nabla F_1\|_{L^2}+\|\nabla V_1\|_{L^2})\nonumber\\
&\le \frac{1}{4}\|\sqrt{\rho}u_t\|_{L^2}^2+C\bar{\rho}\||u||\nabla u|\|_{L^2}^2
+\frac{1}{160\bar{\rho}}(\|\nabla F_1\|_{L^2}^2+(\mu+\mu_r)^2\|\nabla V_1\|_{L^2}^2),\\
J_2&\le C\bar{\rho}^\frac12\||u||\nabla w|\|_{L^2}\|\sqrt{\rho}w_t\|_{L^2}
+C\||u||\nabla w|\|_{L^2}(\|\nabla F_2\|_{L^2}+\|\nabla V_2\|_{L^2})\nonumber\\
&\le \frac{j_I}{4}\|\sqrt{\rho}w_t\|_{L^2}^2+C\bar{\rho}\||u||\nabla w|\|_{L^2}^2
+\frac{1}{160\bar{\rho}}(\|\nabla F_2\|_{L^2}^2+(c_a+c_d)^2\|\nabla V_2\|_{L^2}^2).
\end{align*}
By virtue of \eqref{3.16}, we deduce that
\begin{align}
I_3&\le C\|\nabla\theta\|_{L^2}\|\nabla F_1\|_{L^2}+C\|\rho\theta u\|_{L^2}\|\nabla F_1\|_{L^2}\nonumber\\
&\le C\|\nabla\theta\|_{L^2}\|\nabla F_1\|_{L^2}
+C\bar{\rho}^\frac12\|\rho\|_{L^3}^\frac14\|\sqrt{\rho}\theta\|_{L^2}^\frac12\|\nabla\theta\|_{L^2}^\frac12
\||u||\nabla u|\|_{L^2}^\frac12\|\nabla F_1\|_{L^2}\nonumber\\
&\le \frac{1}{160\bar{\rho}}\|\nabla F_1\|_{L^2}^2+C(\bar{\rho}^2\|\rho\|_{L^3}^\frac12\|\sqrt{\rho}\theta\|_{L^2}
+\bar{\rho})(\|\nabla\theta\|_{L^2}^2+\||u||\nabla u|\|_{L^2}^2).
\end{align}
Noticing that
\begin{align*}
\|\nabla u\|_{L^6}&\le C(\|\curl u\|_{L^6}+\|\divv u\|_{L^6})
\le C(\|V_1\|_{L^6}+\|F_1\|_{L^6}+\|\rho\theta\|_{L^6})\nonumber\\
&\le C(\|\nabla V_1\|_{L^2}+\|\nabla F_1\|_{L^2}+\bar{\rho}\|\nabla\theta\|_{L^2}),\\
\|\nabla w\|_{L^6}&\le C(\|\curl w\|_{L^6}+\|\divv w\|_{L^6})
\le C(\|V_2\|_{L^6}+\|F_2\|_{L^6})\le C(\|\nabla V_2\|_{L^2}+\|\nabla F_2\|_{L^2}),
\end{align*}
which along with H\"older's inequality, Young's inequality, and Gagliardo-Nirenberg inequality yields that
\begin{align*}
I_4&\le C\|\nabla u\|_{L^2}\|\nabla u\|_{L^6}\|F_1\|_{L^3}
+C\|\nabla w\|_{L^2}\|\nabla w\|_{L^6}\|F_1\|_{L^3}\nonumber\\
&\quad+C\|\nabla u\|_{L^2}\|\rho\theta\|_{L^6}\|F_1\|_{L^3}+C\|w\|_{L^2}\|w\|_{L^6}\|F_1\|_{L^3}\nonumber\\
&\le C\|\nabla u\|_{L^2}(\|\nabla V_1\|_{L^2}+\|\nabla F_1\|_{L^2}
+\bar{\rho}\|\nabla\theta\|_{L^2})\|F_1\|_{L^2}^\frac12\|\nabla F_1\|_{L^2}^\frac12\nonumber\\
&\quad+C(\bar{\rho}\|\nabla u\|_{L^2}\|\nabla\theta\|_{L^2}+\|w\|_{L^2}\|\nabla w\|_{L^2})
\|F_1\|_{L^2}^\frac12\|\nabla F_1\|_{L^2}^\frac12\nonumber\\
&\quad+C\|\nabla w\|_{L^2}(\|\nabla V_2\|_{L^2}+\|\nabla F_2\|_{L^2})\|F_1\|_{L^2}^\frac12\|\nabla F_1\|_{L^2}^\frac12\nonumber\\
&\le \frac{1}{160\bar{\rho}}\Big(\|\nabla F_1\|_{L^2}^2+\|\nabla F_2\|_{L^2}^2
+(\mu+\mu_r)^2\|\nabla V_1\|_{L^2}^2+(c_a+c_d)^2\|\nabla V_2\|_{L^2}^2\Big)\nonumber\\
&\quad+C\bar{\rho}^3(\|\nabla u\|_{L^2}^4+\|\nabla w\|_{L^2}^4)\|F_1\|_{L^2}^2+C\bar{\rho}\|\nabla\theta\|_{L^2}^2
+\frac{\varepsilon}{\bar{\rho}^2}\|w\|_{L^2}^2,\\
I_5+I_6&\le \frac{1}{160\bar{\rho}}\big(\|\nabla F_2\|_{L^2}^2
+(c_a+c_d)^2\|\nabla V_2\|_{L^2}^2\big)+\frac{C}{\bar{\rho}^2}\big(\|w\|_{L^2}^2+\|\nabla u\|_{L^2}^2\big).
\end{align*}
Inserting the above estimates for $I_i\ (i=1, 2, \cdots, 6)$ into \eqref{z3.31} gives that
\begin{align}\label{3.36}
&\frac{d}{dt}\Big(\mu\|V_1\|_{L^2}^2
+(c_a+c_d)\|V_2\|_{L^2}^2+\frac{\|F_1\|_{L^2}^2}{2\mu+\lambda}+\frac{\|F_2\|_{L^2}^2}{2c_d+c_0}
+\mu_r\|V_1-2w\|_{L^2}^2\Big)+\frac12\|\sqrt{\rho}u_t\|_{L^2}^2\nonumber\\
&\quad+\frac{j_I}{2}\|\sqrt{\rho}w_t\|_{L^2}^2+\frac{1}{16\bar{\rho}}\big[\|\nabla F_1\|_{L^2}^2+\|\nabla F_2\|_{L^2}^2
+(\mu+\mu_r)^2\|\nabla V_1\|_{L^2}^2+(c_a+c_d)^2\|\nabla V_2\|_{L^2}^2\big]\nonumber\\
&\le C\big(\bar{\rho}^2\|\rho\|_{L^3}^\frac12\|\sqrt{\rho}\theta\|_{L^2}
+\bar{\rho}\big)\big(\|\nabla\theta\|_{L^2}^2+\||u||\nabla u|\|_{L^2}^2+\||u||\nabla w|\|_{L^2}^2\big)\nonumber\\
&\quad+\frac{C}{\bar{\rho}^2}\big(\|w\|_{L^2}^2+\|\nabla u\|_{L^2}^2\big)+C\bar{\rho}^3\big(\|\nabla u\|_{L^2}^4+\|\nabla w\|_{L^2}^4\big)\|F_1\|_{L^2}^2.
\end{align}
Due to
\begin{align*}
&\|\nabla u\|_{L^2}\le C\big(\|V_1\|_{L^2}+\|F_1\|_{L^2}+\|\rho\theta\|_{L^2}\big)\le
 C\big(\|V_1\|_{L^2}+\|F_1\|_{L^2}+\sqrt{\bar{\rho}}\|\sqrt{\rho}\theta\|_{L^2}\big),\\
&\|\nabla w\|_{L^2}\le C\big(\|V_2\|_{L^2}+\|F_2\|_{L^2}\big),
\end{align*}
and
\begin{align*}
\|w\|_{L^2}^2&=\frac14\|V_1-2w-V_1\|_{L^2}^2< \|V_1-2w\|_{L^2}^2+\|V_1\|_{L^2}^2,
\end{align*}
we get \eqref{3.21} after integrating \eqref{3.36} over $[0,T]$.
\end{proof}

To obtain the estimate of $L^\infty(0, T; L^\infty)$ for the density, we need the following result. The detailed proof can be found in \cite[Proposition 2.6]{JL20}, and we omit it for simplicity.
\begin{lemma}\label{l35}
Under the condition \eqref{3.17}, it holds that
\begin{align}
\sup_{0\le t\le T}\|\rho\|_{L^\infty}\le
\|\rho_0\|_{L^\infty}e^{C\bar{\rho}^\frac23\sup\limits_{0\le t\le T}\|\sqrt{\rho}u\|_{L^2}^\frac13\|\sqrt{\rho}|u|^2\|_{L^2}^\frac13
+C\bar{\rho}\int_0^T\|\nabla u\|_{L^2}\|(\nabla F_1, \nabla V_1, \bar{\rho}\nabla\theta)\|_{L^2}dt}.
\end{align}
\end{lemma}

\begin{lemma}\label{l36}
Let
\begin{align*}
N_T\triangleq\bar{\rho}\big[\|\rho\|_{L^3}+\bar{\rho}^2
\big(\|\sqrt{\rho}u\|_{L^2}^2+\|\sqrt{\rho}w\|_{L^2}^2\big)\big]\big(\|\nabla u\|_{L^2}^2+\|w\|_{H^1}^2+
\bar{\rho}\|\sqrt{\rho}\Psi\|_{L^2}^2\big).
\end{align*}
There exists a positive constant $\eta_0$ depending only on $\mu$, $\lambda$, $\mu_r$, $c_0$, $c_a$, $c_d$, $j_I$, $c_v$, $R$, and $\kappa$ such that if
\begin{align}\label{3.41}
\eta\le\eta_0,\quad \sup_{0\le t\le T}\|\rho\|_{L^\infty}\le 4\bar{\rho}, \quad N_T\le \sqrt{\eta},
\end{align}
then it holds that
\begin{align}
&\sup_{0\le t\le T}\|\rho\|_{L^3}+\Big(\int_0^T\int\rho^3pdxdt\Big)^\frac13\le C(\|\rho_0\|_{L^3}
+\bar{\rho}^2(\|\sqrt{\rho_0}u_0\|_{L^2}^2+\|\sqrt{\rho_0}w_0\|_{L^2}^2)),
\label{3.42} \\
&\bar{\rho}^2\Big(\sup_{0\le t\le T}(\|\sqrt{\rho}u\|_{L^2}^2+\|\sqrt{\rho}w\|_{L^2}^2)
+\int_0^T\|(\nabla u, \nabla w, w)\|_{L^2}^2dt\Big)\nonumber\\
&\le C(\|\rho_0\|_{L^3}
+\bar{\rho}^2(\|\sqrt{\rho_0}u_0\|_{L^2}^2
+\|\sqrt{\rho_0}w_0\|_{L^2}^2)), \label{3.43} \\
&\sup_{0\le t\le T}(\|\nabla u\|_{L^2}^2+\|w\|_{H^1}^2+\bar{\rho}\|\sqrt{\rho}\Psi\|_{L^2}^2)
+\int_0^T\Big\|\Big(\sqrt{\rho}u_t, \sqrt{\rho}w_t, \frac{\nabla F_1}{\sqrt{\bar{\rho}}},
 \frac{\nabla F_2}{\sqrt{\bar{\rho}}}, \frac{\nabla V_1}{\sqrt{\bar{\rho}}},
  \frac{\nabla V_2}{\sqrt{\bar{\rho}}}\Big)\Big\|_{L^2}^2dt\nonumber\\
& \quad +\bar{\rho}\int_0^T\|(|w||\nabla u|, |u||\nabla w|, |w||\nabla w|, \nabla\theta, |u||\nabla u|, |u||w|, |w|^2)\|_{L^2}^2dt
\nonumber\\
& \le C(\bar{\rho}\|\sqrt{\rho_0}\Psi_0\|_{L^2}^2+\|\nabla u_0\|_{L^2}^2
+\|w_0\|_{H^1}^2), \label{3.44}  \\
&\sup_{0\le t\le T}\|\rho\|_{L^\infty}\le \bar{\rho}e^{CN_0^\frac16+CN_0^\frac12}. \label{3.45}
\end{align}
\end{lemma}
\begin{proof}[Proof]
1. It follows from \eqref{3.41} and Lemma \ref{l32} that
\begin{align}\label{3.46}
&\sup_{0\le t\le T}\|\sqrt{\rho}\Psi\|_{L^2}^2
+\int_0^T\big(\|\nabla\theta\|_{L^2}^2+\||u||\nabla u|\|_{L^2}^2+\||w||\nabla w|\|_{L^2}^2\big)dt\nonumber\\
&\quad+\int_0^T\big(\||w||\nabla u|\|_{L^2}^2
+\||u||\nabla w|\|_{L^2}^2+\||u||w|\|_{L^2}^2+\|w\|_{L^4}^4\big)dt\nonumber\\
&\le C\|\sqrt{\rho_0}\Psi_0\|_{L^2}^2
+C\eta^\frac14\int_0^T\big(\|\nabla\theta\|_{L^2}^2+\||u||\nabla u|\|_{L^2}^2\big)dt+C\int_0^T\|\nabla u\|_{L^2}^4\|\nabla w\|_{L^2}^2dt\nonumber\\
&\quad+\frac{\varepsilon C}{\bar{\rho}^2}\int_0^T\|w\|_{L^2}^2dt
+C\bar{\rho}^2\int_0^T\|\nabla w\|_{L^2}^4\|\nabla u\|_{L^2}^2dt.
\end{align}
By \eqref{3.41} and \eqref{3.3}, we get that
\begin{align}\label{3.47}
&C\bar{\rho}^3\int_0^T\big(\|\nabla u\|_{L^2}^4\|\nabla w\|_{L^2}^2+\|\nabla w\|_{L^2}^4\|\nabla u\|_{L^2}^2+\|w\|_{L^2}^6\big)dt\nonumber\\
&\le C\sup_{0\le t\le T}\big(\|\nabla u\|_{L^2}^2+\|w\|_{H^1}^2\big)
\sup_{0\le t\le T}\big(\|\nabla u\|_{L^2}^2+\|w\|_{H^1}^2\big)\nonumber\\
&\quad\times\left[\sup_{0\le t\le T}\big(\|\sqrt{\rho}u\|_{L^2}^2
+\|\sqrt{\rho}w\|_{L^2}^2\big)+\sup_{0\le t\le T}\|\rho\|_{L^3}^2\int_0^T\|\nabla\theta\|_{L^2}^2dt\right]\nonumber\\
&\le C\eta^\frac12\sup_{0\le t\le T}\big(\|\nabla u\|_{L^2}^2+\|w\|_{H^1}^2\big)
+C\eta\int_0^T\|\nabla\theta\|_{L^2}^2dt,
\end{align}
and
\begin{align}\label{3.48}
\frac{C\varepsilon }{\bar{\rho}^2}\int_0^T\|w\|_{L^2}^2dt
&\le \frac{C\varepsilon }{\bar{\rho}^2}\left(\|\sqrt{\rho_0}u_0\|_{L^2}^2
+\|\sqrt{\rho_0}w_0\|_{L^2}^2
+\|\rho\|_{L^3}^2\int_0^T\|\nabla\theta\|_{L^2}^2dt\right)\nonumber\\
&\le \frac{C\varepsilon }{\bar{\rho}^\frac53}\big(\|\nabla u_0\|_{L^2}^2+\|\nabla w_0\|_{L^2}^2\big)
+\frac{ C\|\rho_0\|_{L^1}^\frac23\varepsilon }{\bar{\rho}^\frac23}\int_0^T\|\nabla\theta\|_{L^2}^2dt,
\end{align}
owing to
\begin{align*}
\|\sqrt{\rho_0}u_0\|_{L^2}^2+\|\sqrt{\rho_0}w_0\|_{L^2}^2
&\le \|\rho\|_{L^\frac32}\big(\|u_0\|_{L^6}^2+\|w_0\|_{L^6}^2\big) \notag \\
& \le C\|\rho_0\|_{L^\frac32}\big(\|\nabla u_0\|_{L^2}^2+\|\nabla w_0\|_{L^2}^2\big)\nonumber\\
&\le C\|\rho_0\|_{L^1}^\frac23\|\rho_0\|_{L^\infty}^\frac13\big(\|\nabla u_0\|_{L^2}^2+\|\nabla w_0\|_{L^2}^2\big).
\end{align*}
Putting \eqref{3.47} and \eqref{3.48} into \eqref{3.46}, one obtains after choosing $\varepsilon\le \eta$ suitably small that
\begin{align*}
&\sup_{0\le t\le T}\|\sqrt{\rho}\Psi\|_{L^2}^2
+\int_0^T\big(\|\nabla\theta\|_{L^2}^2+\||u||\nabla u|\|_{L^2}^2+\||w||\nabla w|\|_{L^2}^2\big)dt\nonumber\\
&\quad+\int_0^T\big(\||w||\nabla u|\|_{L^2}^2
+\||u||\nabla w|\|_{L^2}^2+\||u||w|\|_{L^2}^2\big)dt\nonumber\\
&\le C\Big[\|\sqrt{\rho_0}\Psi_0\|_{L^2}^2+\frac{1}{\bar{\rho}^\frac53}\big(\|\nabla u_0\|_{L^2}^2+\|\nabla w_0\|_{L^2}^2\big)\Big]
+\frac{C\eta^\frac12}{\bar{\rho}}\sup_{0\le t\le T}\big(\|\nabla u\|_{L^2}^2+\|w\|_{H^1}^2\big),
\end{align*}
provided that $\eta_0$ is small enough, which yields that
\begin{align}\label{3.50}
&\bar{\rho}\sup_{0\le t\le T}\|\sqrt{\rho}\Psi\|_{L^2}^2
+\bar{\rho}\int_0^T\big(\|\nabla\theta\|_{L^2}^2+\||u||\nabla u|\|_{L^2}^2+\||w||\nabla w|\|_{L^2}^2\big)dt\nonumber\\
&\quad+\bar{\rho}\int_0^T\big(\||w||\nabla u|\|_{L^2}^2
+\||u||\nabla w|\|_{L^2}^2+\||u||w|\|_{L^2}^2\big)dt\nonumber\\
&\le C\Big[\bar{\rho}\|\sqrt{\rho_0}\Psi_0\|_{L^2}^2+\frac{1}{\bar{\rho}^\frac23}\big(\|\nabla u_0\|_{L^2}^2+\|\nabla w_0\|_{L^2}^2\big)\Big]
+C\eta^\frac12\sup_{0\le t\le T}\big(\|\nabla u\|_{L^2}^2+\|w\|_{H^1}^2\big).
\end{align}

2. We infer from \eqref{3.21} and \eqref{3.41} that
\begin{align}\label{3.51}
&\sup_{0\le t\le T}\big(\|\nabla u\|_{L^2}^2+\|w\|_{H^1}^2\big)
+\int_0^T\Big\|\Big(\sqrt{\rho}u_t, \sqrt{\rho}w_t, \frac{\nabla F_1}{\sqrt{\bar{\rho}}},
 \frac{\nabla F_2}{\sqrt{\bar{\rho}}}, \frac{\nabla V_1}{\sqrt{\bar{\rho}}},
  \frac{\nabla V_2}{\sqrt{\bar{\rho}}}\Big)\Big\|_{L^2}^2dt\nonumber\\
&\le C\big(\|\nabla u_0\|_{L^2}^2+\|w_0\|_{H^1}^2\big)
+C\bar{\rho}^3\int_0^T\big(\|\nabla u\|_{L^2}^4+\|\nabla w\|_{L^2}^4\big)\big(\|\nabla u\|_{L^2}^2
+\bar{\rho}\|\sqrt{\rho}\Psi\|_{L^2}^2\big)dt\nonumber\\
&\quad+C\int_0^T\big(\bar{\rho}^2\|\rho\|_{L^3}^\frac12\|\sqrt{\rho}\theta\|_{L^2}
+\bar{\rho}\big)\big(\|\nabla\theta\|_{L^2}^2+\||u||\nabla u|\|_{L^2}^2+\||u||\nabla w|\|_{L^2}^2\big)dt\nonumber\\
&\quad+C\bar{\rho}\sup_{0\le t\le T}\|\sqrt{\rho}\Psi\|_{L^2}^2
+\frac{C}{\bar{\rho}}\int_0^T\big(\|w\|_{L^2}^2+\|\nabla u\|_{L^2}^2\big)dt\nonumber\\
&\le C\big(\|\nabla u_0\|_{L^2}^2+\|w_0\|_{H^1}^2+\bar{\rho}\|\sqrt{\rho_0}\Psi_0\|_{L^2}^2\big)+C\eta^\frac12\sup_{0\le t\le T}\big(\|\nabla u\|_{L^2}^2+\|w\|_{H^1}^2\big)\nonumber\\
&\quad+C\eta^\frac14\int_0^T
\big(\|\nabla\theta\|_{L^2}^2+\||u||\nabla u|\|_{L^2}^2+\||u||\nabla w|\|_{L^2}^2\big)dt
+C\bar{\rho}\eta\int_0^T\|\nabla\theta\|_{L^2}^2dt,
\end{align}
where we have used
\begin{align*}
&C\bar{\rho}^3\int_0^T\big(\|\nabla u\|_{L^2}^4+\|\nabla w\|_{L^2}^4\big)\big(\|\nabla u\|_{L^2}^2
+\bar{\rho}\|\sqrt{\rho}\Psi\|_{L^2}^2\big)dt\nonumber\\
&\le C\bar{\rho}^3\sup_{0\le t\le T}\big(\|\nabla u\|_{L^2}^2
+\bar{\rho}\|\sqrt{\rho}\Psi\|_{L^2}^2\big)\sup_{0\le t\le T}\big(\|\nabla u\|_{L^2}^2+\|\nabla w\|_{L^2}^2\big)\nonumber\\
&\quad\times\Big(\sup_{0\le t\le T}\big(\|\sqrt{\rho}u\|_{L^2}^2+\|\sqrt{\rho}w\|_{L^2}^2\big)
+\|\rho\|_{L^3}^2\int_0^T\|\nabla\theta\|_{L^2}^2dt\Big)\nonumber\\
&\le C\eta^\frac12\sup_{0\le t\le T}(\|\nabla u\|_{L^2}^2+\|\nabla w\|_{L^2}^2)
+C\bar{\rho}\eta\int_0^T\|\nabla\theta\|_{L^2}^2dt,
\end{align*}
and
\begin{align*}
\frac{C}{\bar{\rho}}\int_0^T(\|w\|_{L^2}^2+\|\nabla u\|_{L^2}^2)dt
&\le \frac{C}{\bar{\rho}^2}\Big(\|\sqrt{\rho_0}u_0\|_{L^2}^2+\|\sqrt{\rho_0}w_0\|_{L^2}^2+\|\rho\|_{L^3}^2\int_0^T\|\nabla\theta\|_{L^2}^2dt\Big)\nonumber\\
&\le \frac{C}{\bar{\rho}^\frac23}(\|\nabla u_0\|_{L^2}^2+\|\nabla w_0\|_{L^2}^2)
+C\|\rho_0\|_{L^1}^\frac23\bar{\rho}^\frac13\int_0^T\|\nabla\theta\|_{L^2}^2dt\nonumber\\
&\le C\Big(\bar{\rho}\|\sqrt{\rho_0}\Psi_0\|_{L^2}^2+\frac{1}{\bar{\rho}^\frac23}(\|\nabla u_0\|_{L^2}^2+\|\nabla w_0\|_{L^2}^2)\Big)\nonumber\\
&\quad+C\eta^\frac12\sup_{0\le t\le T}(\|\nabla u\|_{L^2}^2+\|w\|_{H^1}^2).
\end{align*}
Combining \eqref{3.50} and \eqref{3.51}, we then obtain \eqref{3.44} after choosing $\eta_0$ sufficiently small.

3. With the help of \eqref{3.44}, we derive from \eqref{3.3} and \eqref{3.41} that
\begin{align}\label{3.55}
&\sup_{0\le t\le T}\big(\|\sqrt{\rho}u\|_{L^2}^2+\|\sqrt{\rho}w\|_{L^2}^2\big)+
\int_0^T\big(\|\nabla u\|_{L^2}^2+\|w\|_{H^1}^2\big)dt\nonumber\\
&\le C\big(\|\sqrt{\rho_0}u_0\|_{L^2}^2+\|\sqrt{\rho_0}w_0\|_{L^2}^2\big)
+\frac{C}{\bar{\rho}}\sup_{0\le t\le T}\|\rho\|_{L^3}^2\big(\|\nabla u_0\|_{L^2}^2+\|w_0\|_{H^1}^2+\bar{\rho}\|\sqrt{\rho_0}\Psi_0\|_{L^2}^2\big)\nonumber\\
&\le  C\big(\|\sqrt{\rho_0}u_0\|_{L^2}^2+\|\sqrt{\rho_0}w_0\|_{L^2}^2\big)+\frac{C}{\bar{\rho}}\sup_{0\le t\le T}\|\rho\|_{L^3}^2
\sup_{0\le t\le T}\big(\|\nabla u\|_{L^2}^2+\|w\|_{H^1}^2+\bar{\rho}\|\sqrt{\rho}\Psi\|_{L^2}^2\big)\nonumber\\
&\le C\big(\|\sqrt{\rho_0}u_0\|_{L^2}^2+\|\sqrt{\rho_0}w_0\|_{L^2}^2\big)
+C\eta^\frac12\frac{1}{\bar{\rho}^2}\sup_{0\le t\le T}\|\rho\|_{L^3}.
\end{align}
This together with \eqref{3.18}, \eqref{3.41}, H\"older's, Young's, and Gagliardo-Nirenberg inequalities gives that
 \begin{align}\label{3.56}
&\sup_{0\le t\le T}\|\rho\|_{L^3}^3+\int_0^T\int\rho^3pdxdt\nonumber\\
 &\le \|\rho_0\|_{L^3}^3
+C\bar{\rho}^\frac23\sup_{0\le t\le T}(\|\sqrt{\rho}u\|_{L^2}^\frac13
\|\sqrt{\rho}\Psi\|_{L^2}^\frac13\|\rho\|_{L^3}^3)
+C\bar{\rho}^2\int_0^T(\|\rho\|_{L^3}^2\|\nabla u\|_{L^2}^2)dt\nonumber\\
&\le \|\rho_0\|_{L^3}^3
+C(\eta^\frac{1}{12}+\eta^\frac12)\sup_{0\le t\le T}\|\rho\|_{L^3}^3+C\bar{\rho}^2\sup_{0\le t\le T}\|\rho\|_{L^3}^2
(\|\sqrt{\rho_0}u_0\|_{L^2}^2+\|\sqrt{\rho_0}w_0\|_{L^2}^2)\nonumber\\
&\le \|\rho_0\|_{L^3}^3
+C\Big(\eta^\frac{1}{12}+\eta^\frac12+\frac14\Big)\sup_{0\le t\le T}\|\rho\|_{L^3}^3
+C\bar{\rho}^6(\|\sqrt{\rho_0}u_0\|_{L^2}^2+\|\sqrt{\rho_0}w_0\|_{L^2}^2)^3,
 \end{align}
which implies \eqref{3.42} by choosing $\eta_0$ suitably small.
Moreover, we get from \eqref{3.55} and \eqref{3.56} that
\begin{align}\label{3.57}
&\bar{\rho}^2\sup_{0\le t\le T}(\|\sqrt{\rho}u\|_{L^2}^2+\|\sqrt{\rho}w\|_{L^2}^2)+
\bar{\rho}^2\int_0^T(\|\nabla u\|_{L^2}^2+\|w\|_{H^1}^2)dt\nonumber\\
&\le C\bar{\rho}^2(\|\sqrt{\rho_0}u_0\|_{L^2}^2+\|\sqrt{\rho_0}w_0\|_{L^2}^2)
+C\eta^\frac12\sup_{0\le t\le T}\|\rho\|_{L^3}\nonumber\\
&\le C(\|\rho_0\|_{L^3}+\bar{\rho}^2(\|\sqrt{\rho_0}u_0\|_{L^2}^2+\|\sqrt{\rho_0}w_0\|_{L^2}^2)),
\end{align}
which yields \eqref{3.43}. Finally, \eqref{3.45} follows from Lemma \ref{l35}, \eqref{3.42}, \eqref{3.44}, and \eqref{3.57}.
\end{proof}

\begin{lemma}\label{l37}
Let \eqref{1.5} be satisfied. Let $\eta_0$, $N_T$, and $N_0$ be as in Lemma \ref{l36}.
There exists a number $\varepsilon_0\in (0, \eta_0)$ such that if
\begin{align}\label{3.58}
\sup_{0\le t\le T}\|\rho\|_{L^\infty}\le 4\bar{\rho}, \quad N_T\le \sqrt{\varepsilon_0},\quad and\quad N_0\le \varepsilon_0,
\end{align}
then
\begin{align}
\sup_{0\le t\le T}\|\rho\|_{L^\infty}\le 2\bar{\rho} \quad and\quad N_T\le \frac{\sqrt{\varepsilon_0}}{2},
\end{align}
where $\varepsilon_0$ depends only on  $\mu$, $\lambda$, $\mu_r$, $c_0$, $c_a$, $c_d$, $j_I$, $c_v$, $R$, and $\kappa$.
\end{lemma}
\begin{proof}[Proof]
Let $\varepsilon_0\in(0, \eta_0)$, then we deduce from \eqref{3.58} and Lemma \ref{l36} that
\begin{align*}
N_T\le\bar{\rho}\big[\|\rho_0\|_{L^3}+\bar{\rho}^2\big(\|\sqrt{\rho_0}u_0\|_{L^2}^2
+\|\sqrt{\rho_0}w_0\|_{L^2}^2\big)\big]\big(\|\nabla u_0\|_{L^2}^2+\|w_0\|_{H^1}^2+
\bar{\rho}\|\sqrt{\rho_0}\Psi_0\|_{L^2}^2\big)\le C\varepsilon_0\le \frac{\sqrt{\varepsilon_0}}{2},
\end{align*}
and
\begin{align*}
\sup_{0\le t\le T}\|\rho\|_{L^\infty}\le \bar{\rho}e^{CN_0^\frac16+CN_0^\frac12}\le \bar{\rho}e^{C\varepsilon_0^\frac16+C\varepsilon_0^\frac12}
\le 2\bar{\rho},
\end{align*}
provided that $\varepsilon_0$ is sufficiently small.
\end{proof}

\begin{corollary}\label{c31}
Let \eqref{1.5} be satisfied. Let $\varepsilon_0$ be as in Lemma \ref{l37} and
assume $N_0\le \varepsilon_0$. Then there exists a positive constant $C$ depending only on $\mu$, $\lambda$, $\mu_r$, $c_0$,
 $c_a$, $c_d$, $j_I$, $c_v$, $R$, $\kappa$, $\|\rho_0\|_{L^1}$, $\|\rho_0\|_{L^3}$, $\|\sqrt{\rho_0}u_0\|_{L^2}$, $\|\sqrt{\rho_0}w_0\|_{L^2}$,
 $\|\sqrt{\rho_0}\Psi_0\|_{L^2}$, $\|\nabla u_0\|_{L^2}$, and $\|w_0\|_{H^1}$, such that the following estimates hold true
 \begin{align*}
 &\sup_{0\le t\le T}\big(\|(\sqrt{\rho}\Psi, \sqrt{\rho}u, \sqrt{\rho}w, \sqrt{\rho}uw,\nabla u, \nabla w, w)\|_{L^2}^2+\|\rho\|_{L^3}+\|\rho\|_{L^\infty}\big)\le C,\\
 & \int_0^T\big(\|(\nabla\theta, |w||\nabla u|, |u||\nabla w|, |u||\nabla u|, |w||\nabla w|, \sqrt{\rho}w_t, \sqrt{\rho}u_t,
  \nabla F_1, \nabla V_1)\|_{L^2}^2\nonumber\\
  &\quad+\|\nabla u\|_{L^6}^2+\|w\|_{H^2}^2+\int\rho^3pdx\big)dt\le C.
 \end{align*}
\end{corollary}

\section{Time dependent higher order estimates}\label{sec4}
In this section, by virtue of Corollary \ref{c31}, we can derive some higher order estimates on $(\rho, u, w, \theta)$, which is sufficient to ensure the global existence of strong solutions. In what follows, we mention that positive constant $C$ may depend on $\mu$, $\lambda$, $\mu_r$, $c_0$, $c_a$, $c_d$, $j_I$, $c_v$, $R$, $\kappa$, $\|\rho_0\|_{L^1}$, $\Psi_0$, $g_1$, $g_2$, $g_3$, and $T$.
\begin{lemma}\label{l41}
Let $\varepsilon_0$ be as in Lemma \ref{l37} and assume that $N_0\le \varepsilon_0$, then it holds that
\begin{align}\label{4.1}
\sup_{0\le t\le T}\big(\|\nabla\theta\|_{L^2}^2+\|\sqrt{\rho}\dot{u}\|_{L^2}^2
+\|\sqrt{\rho}\dot{w}\|_{L^2}^2\big)
+\int_0^T\|(\sqrt{\rho}\dot{\theta}, \nabla\dot{u}, \nabla\dot{w}, \dot{w})\|_{L^2}^2dt\le C(T).
\end{align}
\end{lemma}
\begin{proof}

1. Applying $\dot{u}^j(\partial_t+\divv(u\cdot))$ to the $j$-th component of $\eqref{a1}_{2}$ and $\dot{w}^j(\partial_t+\divv(u\cdot))$ to the $j$-th component of $\eqref{a1}_{3}$, respectively, and then summing up, one gets by some calculations that
\begin{align}\label{f30}
&\frac{1}{2}\frac{d}{dt}\int\big(\rho|\dot{u}|^2+j_I\rho|\dot{w}|^2\big)dx\nonumber\\
&=-\int\dot{u}^j\big(\partial_jp_t
+\divv(u\partial_jp)\big)dx+(\mu+\mu_r)\int\dot{u}^j\big(\partial_t\Delta u^j
+\divv(u\Delta u^j)\big)dx
\nonumber\\
&\quad+(\mu+\lambda-\mu_r)\int\dot{u}^j\big(\partial_t\partial_j\divv u
+\divv(u\partial_j\divv u)\big)dx+(c_a+c_d)\int\dot{w}^j\big(\Delta w_t^j
+\divv(u\Delta w^j)\big)dx\nonumber\\
&\quad+(c_0+c_d-c_a)\int\dot{w}^j\big(\partial_j\divv w_t
+\divv(u\partial_j\divv w)\big)dx+2\mu_r\int\dot{u}
\cdot\big(\curl w_t+\partial_j(u^j\curl w)\big)dx\nonumber\\
&\quad
+2\mu_r\int\dot{w}\cdot\big(\curl u_t+\partial_j(u^j\curl u)\big)dx
-4\mu_r\int\dot{w}^j\big(w_t^j+\divv(w^ju)\big)dx\triangleq\sum_{i=1}^8K_i.
\end{align}
Integration by parts together with $\eqref{a1}_1$, H\"older's inequality, Corollary \ref{c31}, and Young's inequality yields that
\begin{align*}
K_1&=\int(\partial_j\dot{u}^jp_t+\partial_j pu\cdot\nabla\dot{u}^j)dx\nonumber\\
&=\int\partial_j\dot{u}^jp_tdx+\int \partial_j p(u\cdot\nabla\dot{u}^j)dx\nonumber\\
&=R\int\partial_j\dot{u}^j\big[(\rho\theta)_t+{\rm div}(\rho u\theta)\big]dx
-R\int\rho\theta\partial_ju\cdot\nabla\dot{u}^jdx\nonumber\\
&=R\int\rho\dot{\theta}\partial_j\dot{u}^jdx-R\int\rho\theta\partial_ju\cdot\nabla\dot{u}^jdx\nonumber\\
&\le \frac{\mu}{6}\|\nabla\dot{u}\|_{L^2}^2+C\|\rho\theta\nabla u\|_{L^2}^2+C\|\sqrt{\rho}\dot{\theta}\|_{L^2}^2\nonumber\\
&\le \frac{\mu}{6}\|\nabla\dot{u}\|_{L^2}^2+C\|\rho\theta\|_{L^2}^\frac12\|\theta\|_{L^6}^\frac32\|\nabla u\|_{L^4}^2
+C\|\sqrt{\rho}\dot{\theta}\|_{L^2}^2\nonumber\\
&\le \frac{\mu}{6}\|\nabla\dot{u}\|_{L^2}^2
+C\big(1+\|\sqrt{\rho}\dot{\theta}\|_{L^2}^2+\|\nabla\theta\|_{L^2}^4+\|\nabla u\|_{L^4}^4\big).
\end{align*}
It follows from integration by parts that
\begin{align*}
K_2+K_3&=(\mu+\mu_r)\int\dot{u}^j(\Delta(\dot{u}^j-u\cdot\nabla u^j)
+\divv (u\Delta u^j))dx\nonumber\\
&\quad+(\mu+\lambda-\mu_r)\int\dot{u}^j(\partial_j{\rm div}(\dot{u}-u\cdot\nabla u)+\divv(u\partial_j\divv u))dx\nonumber\\
&=-(\mu+\mu_r)\int|\nabla\dot{u}|^2dx-(\mu+\lambda-\mu_r)
\int(\divv \dot{u})^2dx\nonumber\\
&\quad-(\mu+\mu_r)\int\dot{u}\Delta(u\cdot\nabla u^j)dx+(\mu+\mu_r)\int\dot{u}^j\divv (u\cdot\Delta u^j)dx\nonumber\\
&\quad-(\mu+\lambda-\mu_r)\int\dot{u}^j\partial_j\divv(u\cdot\nabla u)dx
+(\mu+\lambda-\mu_r)\int\dot{u}^j\divv (u\partial_j\divv u)dx\nonumber\\
&=-\mu \int |\nabla \dot{u}|^{2}dx-(\mu+\lambda)
\int(\divv \dot{u})^{2}dx-\mu_{r} \int |\curl\dot{u}|^{2}dx \nonumber\\
&\quad-(\mu+\mu_{r}) \int(\partial_{i} \dot{u}^{j} \partial_{i} u^{j}{\rm div}u-\partial_{i} \dot{u}^{j} \partial_{i} u^{k} \partial_{k} u^{j}-\partial_{k} \dot{u}^{j} \partial_{i} u^{j} \partial_{i} u^{k}) dx \nonumber\\
&\quad-(\mu+\lambda-\mu_{r}) \int\big(\partial_{k} \dot{u}^{j} \partial_{i} u^{i} \partial_{j} u^{k}-\partial_{k} \dot{u}^{j} \partial_{j} u^{i} \partial_{i} u^{k}-\partial_{k} \dot{u}^{j} \partial_{i} u^{k} \partial_{j} u^{i}\big)dx \nonumber\\
&\le-\mu \int |\nabla \dot{u}|^{2}dx-(\mu+\lambda) \int(\divv \dot{u})^{2} dx-\mu_{r} \int |\curl\dot{u}|^{2}dx+\frac{\mu}{6}\|\nabla \dot{u}\|_{L^{2}}^{2}+C\|\nabla u\|_{L^{4}}^{4}.
\end{align*}
We obtain from integration by parts and H\"older's inequality that
\begin{align*}
K_4+K_5&=(c_{a}+c_{d}) \int \dot{w}^{j}(\Delta(\dot{w}^{j}-u \cdot \nabla w^{j})+\divv (u \Delta w^{j}))dx \\
&\quad+(c_{0}+c_{d}-c_{a}) \int \dot{w}^{j}(\partial_{j} \divv (\dot{w}-u \cdot \nabla w)+\divv(u \partial_{j} \divv w))dx \\
&=-(c_{a}+c_{d}) \int |\nabla \dot{w}|^{2}dx-(c_{0}+c_{d}-c_{a}) \int (\divv\dot{w})^{2}dx \\
&\quad-(c_{a}+c_{d}) \int\dot{w}^{j} \Delta(u \cdot \nabla w^{j})dx+(c_{a}+c_{d}) \int\dot{w}^{j}\divv(u \Delta w^{j})dx \\
&\quad-(c_{0}+c_{d}-c_{a})\int\dot{w}^{j}\partial_{j}\divv(u\cdot \nabla w)dx+(c_{0}+c_{d}-c_{a})
 \int\dot{w}^{j}\divv (u \partial_{j}\divv w)dx\\
&=-c_{d} \int|\nabla \dot{w}|^{2}dx-(c_{0}+c_{d})
\int (\divv \dot{w})^{2}dx-c_{a} \int|\curl\dot{w}|^{2} dx \\
&\quad-(c_{a}+c_{d}) \int\big(\partial_{i} \dot{w}^{j} \partial_{k} u^{k} \partial_{i} w^{j}-\partial_{i} \dot{w}^{j} \partial_{i} u^{k} \partial_{k} w^{j}-\partial_{i} \dot{w}^{j} \partial_{k} u^{i} \partial_{k} w^{j}\big)dx \\
&\quad-(c_{0}+c_{d}-c_{a}) \int \big(\partial_{k} \dot{w}^{j} \partial_{i} u^{i} \partial_{j} w^{k}-\partial_{k} \dot{w}^{j} \partial_{j} u^{i} \partial_{i} w^{k}-\partial_{k} \dot{w}^{j} \partial_{i} u^{k} \partial_{j} w^{i}\big)dx \\
&\le -c_{d} \int|\nabla \dot{w}|^{2} dx-(c_{0}+c_{d}) \int (\divv \dot{w})^2dx
-c_{a} \int |\curl\dot{w}|^{2} dx \\
&\quad+\frac{c_d}{4} \|\nabla \dot{w}\|_{L^{2}}^{2}+C\big(\|\nabla u\|_{L^{4}}^{4}+\|\nabla w\|_{L^{4}}^{4}\big).
\end{align*}
Integrating by parts and applying H\"older's inequality, we arrive at
\begin{align*}
K_6&=2 \mu_{r} \int\dot{u} \cdot\big[\curl w_{t}+\partial_{j}(u^{j}\curl w)\big] dx+2 \mu_{r}
\int\dot{w} \cdot\big[\curl u_{t}+\partial_{j}(u^{j}\curl u)\big]dx \\
& =2 \mu_{r}\int(\dot{u} \cdot \curl \dot{w}+\dot{w} \cdot \curl\dot{u}) dx-2 \mu_{r} \int \dot{u} \cdot\curl (u \cdot \nabla w) dx+2 \mu_{r} \int \dot{u} \cdot \partial_{j}(u^{j}\curl  w)dx \\
&\quad-2 \mu_{r} \int \dot{w} \cdot \curl (u \cdot \nabla u)dx+2 \mu_{r} \int \dot{w} \cdot \partial_{j}(u^{j} \curl u) dx \\
&= 2 \mu_{r} \int(\dot{u} \cdot \curl \dot{w}+\dot{w} \cdot \curl \dot{u})dx-2 \mu_{r} \int \dot{u} \cdot \curl (u \cdot \nabla w)dx-2 \mu_{r} \int  \partial_{j} \dot{u} \cdot(u^{j}\curl  w)dx \\
&\quad-2 \mu_{r} \int \dot{w} \cdot \curl (u \cdot \nabla u) dx-2 \mu_{r} \int \partial_{j} \dot{w} \cdot(u^{j}\curl  u) dx \\
& \le  2 \mu_{r}\int\big(\dot{u} \cdot \curl \dot{w}+\dot{w} \cdot\curl \dot{u}\big)dx+C \|\nabla \dot{u}\|_{L^{2}}\|\nabla w\|_{L^{3}}\|u\|_{L^{6}}+C\|\nabla \dot{w}\|_{L^{2}}\|\nabla u\|_{L^{3}}\|u\|_{L^{6}} \\
 &\le 2 \mu_{r} \int\big(\dot{u} \cdot \curl \dot{w}+\dot{w} \cdot\curl \dot{u}\big)dx+\delta \big(\|\nabla \dot{u}\|_{L^{2}}^{2}+\|\nabla \dot{w}\|_{L^{2}}^{2}\big)+C \big(\|\nabla u\|_{L^{3}}^{3}+\|\nabla w\|_{L^{3}}^{3}+\|\nabla u\|_{L^{2}}^{6}\big) \\
&\le 4\mu_{r} \int\dot{w} \cdot \curl \dot{u}dx+\frac{\mu}{6}\|\nabla \dot{u}\|_{L^{2}}^{2}+\frac{c_d}{4}\|\nabla \dot{w}\|_{L^{2}}^{2}+C\|\nabla u\|_{L^{2}}^{6}+C\big(\|\nabla u\|_{L^{3}}^{3}+\|\nabla w\|_{L^{3}}^{3}\big).
\end{align*}
Based on H\"older's inequality and integration by parts, we have
\begin{align*}
K_7&=-4 \mu_{r} \int \dot{w}^{j}\big[\dot{w}^{j}-u \cdot \nabla w^{j}+\divv(u w^{j})\big] dx \\
&=-4 \mu_{r} \int|\dot{w}|^{2}dx+4 \mu_{r} \int \dot{w}\big[u \cdot \nabla w^{j}-\divv(u w^{j})\big] dx \\
&=-4 \mu_{r} \int |\dot{w}|^{2} dx-4 \mu_{r} \int  \dot{w}^{j} w^{j} \divv u dx \\
&\le-4 \mu_{r} \int |\dot{w}|^{2} dx+\delta\|\dot{w}\|_{L^{2}}^{2}+C\big(\|\nabla u\|_{L^4}^4+\|w\|_{L^4}^4\big).
\end{align*}
Substituting the above estimates for $K_i\ (i=1, 2, \cdots, 7)$ into \eqref{f30} gives that
\begin{align*}
&\frac{1}{2}\frac{d}{dt}\big(\|\sqrt{\rho}\dot{u}\|_{L^2}^2
+j_I\|\sqrt{\rho}\dot{w}\|_{L^2}^2\big)+\frac{\mu}{2}\|\nabla\dot{u}\|_{L^2}^2
+(\mu+\lambda)\|\divv\dot{u}\|_{L^2}^2+\frac{c_d}{2}\|\nabla\dot{w}\|_{L^2}^2\nonumber\\
&\quad+(c_0+c_d)\|\divv\dot{w}\|_{L^2}^2
+c_a\|\curl\dot{w}\|_{L^2}^2+4\mu_r\Big\|\frac12\curl
\dot{u}-\dot{w}\Big\|_{L^2}^2\nonumber\\
&\le \delta\|\dot{w}\|_{L^2}^2+C\|\nabla u\|_{L^{3}}^{3}+\|\nabla w\|_{L^{3}}^{3}
+C\big(\|\nabla u\|_{L^{4}}^{4}+\|\nabla w\|_{L^{4}}^{4}+\|w\|_{L^4}^4\big)\nonumber\\
&\quad+C\big(1+\|\sqrt{\rho}\dot{\theta}\|_{L^2}^2+\|\nabla\theta\|_{L^2}^4\big)
+C\|\nabla u\|_{L^2}^{6}\nonumber\\
&\le \delta\big(\|\nabla\dot{u}\|_{L^2}^2+\|\curl\dot{u}-2\dot{w}\|_{L^2}^2\big)
+C\|\nabla u\|_{L^{3}}^{3}+\|\nabla w\|_{L^{3}}^{3}+C\|\nabla u\|_{L^2}^{6}
\nonumber\\
&\quad+C\big(\|\nabla u\|_{L^{4}}^{4}+\|\nabla w\|_{L^{4}}^{4}+\|w\|_{L^4}^4\big)
+C\|\sqrt{\rho}\dot{\theta}\|_{L^2}^2+C\|\nabla\theta\|_{L^2}^4+C,
\end{align*}
which leads to
\begin{align}\label{4.5}
&\frac{1}{2}\frac{d}{dt}\big(\|\sqrt{\rho}\dot{u}\|_{L^2}^2
+j_I\|\sqrt{\rho}\dot{w}\|_{L^2}^2\big)+\frac{\mu}{4}\|\nabla\dot{u}\|_{L^2}^2
+(\mu+\lambda)\|\divv\dot{u}\|_{L^2}^2+\frac{c_d}{2}\|\nabla\dot{w}\|_{L^2}^2\nonumber\\
&\quad+(c_0+c_d)\|\divv\dot{w}\|_{L^2}^2
+c_a\|\curl\dot{w}\|_{L^2}^2+2\mu_r\Big\|\frac12\curl
\dot{u}-\dot{w}\Big\|_{L^2}^2\nonumber\\
&\le C\big(\|\nabla u\|_{L^4}^4+\|\nabla w\|_{L^4}^4+\|w\|_{L^4}^4\big)+C\|\nabla u\|_{L^2}^2+C\|\nabla w\|_{L^2}^2+C\|\nabla u\|_{L^2}^{6}
\nonumber\\
&\quad+C\|\sqrt{\rho}\dot{\theta}\|_{L^2}^2+C\|\nabla\theta\|_{L^2}^4+C,
\end{align}
due to
\begin{align*}
\|\dot{w}\|_{L^2}^2=\frac14\|\curl \dot{u}-2\dot{w}-\curl \dot{u}\|_{L^2}^2
\le C\|\curl \dot{u}-2\dot{w}\|_{L^2}^2+C\|\curl \dot{u}\|_{L^2}^2.
\end{align*}

2. Multiplying $\eqref{a1}_3$ by $\dot{\theta}$ and integrating the resulting equation by parts yield that
\begin{align}\label{f21}
&\frac{\kappa}{2}\frac{d}{dt}\int|\nabla\theta|^2dx+c_v\int\rho|\dot{\theta}|^2dx\nonumber\\
&=\kappa\int u\cdot\nabla\theta\Delta\theta dx-\int\rho\theta\divv u\dot{\theta}dx+\lambda\int({\rm div}u)^2\dot{\theta}dx
+2\mu\int\mathcal{D}:\mathcal{D}\dot{\theta}dx\nonumber\\
&\quad+4\mu\int\Big|\frac12\curl u-w\Big|^2\dot{\theta}dx
+c_0\int(\divv w)^2\dot{\theta}dx
+(c_a+c_d)\int\nabla w:\nabla w^T\dot{\theta}dx\nonumber\\
&\quad+(c_d-c_a)\nabla w:\nabla w\dot{\theta}dx\triangleq\sum_{i=1}^8J_i.
\end{align}
Let us deal with the right-hand side terms of \eqref{f21}.  By the standard $L^2$-estimate of elliptic system, one derives that
\begin{align*}
\|\nabla^2\theta\|_{L^2}&\le C(\|\rho\dot{\theta}\|_{L^2}
+\|\rho\theta\divv u\|_{L^2}+\|\nabla u\|_{L^4}^2
+\|w\|_{L^4}^2+\|\nabla w\|_{L^4}^2)\nonumber\\[3pt]
&\le C\|\sqrt{\rho}\dot{\theta}\|_{L^2}+C\bar{\rho}\|\nabla u\|_{L^2}\|\theta\|_{L^\infty}
+C\|\nabla u\|_{L^4}^2
+C\|w\|_{L^4}^2+C\|\nabla w\|_{L^4}^2\nonumber\\
&\le \frac12\|\nabla^2\theta\|_{L^2}+C(\|\sqrt{\rho}\dot{\theta}\|_{L^2}
+\|\nabla u\|_{L^4}^2
+\|w\|_{L^4}^2+\|\nabla w\|_{L^4}^2),
\end{align*}
which implies that
\begin{align}\label{f22}
\|\nabla^2\theta\|_{L^2}\le C(\|\sqrt{\rho}\dot{\theta}\|_{L^2}
+\|\nabla u\|_{L^4}^2
+\|w\|_{L^4}^2+\|\nabla w\|_{L^4}^2).
\end{align}
Thus, we obtain from Sobolev's inequality, Gagliardo-Nirenberg inequality, and \eqref{f22} that
\begin{align*}
J_1&=\kappa\int\Big(\frac{1}{2}|\nabla\theta|^2{\rm div}u-\partial_i\theta\partial_iu^k\partial_k\theta\Big)dx\nonumber\\
&\le C\|\nabla u\|_{L^2}\|\nabla\theta\|_{L^2}^\frac{1}{2}\|\nabla^2\theta\|_{L^2}^\frac{3}{2}\nonumber\\
&\le C\|\nabla u\|_{L^2}\|\nabla\theta\|_{L^2}^\frac{1}{2}\Big(\|\sqrt{\rho}\dot{\theta}\|_{L^2}
+\|\nabla u\|_{L^4}^2
+\|w\|_{L^4}^2+\|\nabla w\|_{L^4}^2\Big)^\frac32\nonumber\\
&\le \delta(\|\sqrt{\rho}\dot{\theta}\|_{L^2}^2+\|\nabla\theta\|_{L^2}^2+\|\nabla u\|_{L^4}^4
+\|\nabla w\|_{L^4}^4+\|w\|_{L^4}^4)+C(\delta)\|\nabla u\|_{L^2}^4\|\nabla\theta\|_{L^2}^2\nonumber\\
&\le \delta\|\sqrt{\rho}\dot{\theta}\|_{L^2}^2+C(\|\nabla\theta\|_{L^2}^2
+\|w\|_{L^4}^4+\|\nabla u\|_{L^4}^4+\|\nabla w\|_{L^4}^4).
\end{align*}
It is easy to see that
\begin{align*}
J_2\le C\|\sqrt{\rho}\dot{\theta}\|_{L^2}\|\sqrt{\rho}\theta\|_{L^2}^\frac14\|\theta\|_{L^6}^\frac34\|\nabla u\|_{L^4}
\le \delta\|\sqrt{\rho}\dot{\theta}\|_{L^2}^2
+C(\delta)(1+\|\nabla\theta\|_{L^2}^4+\|\nabla u\|_{L^4}^4).
\end{align*}
Integration by parts together with H\"older's inequality leads to
\begin{align}\label{f25}
J_3&=\lambda\int({\rm div}u)^2\theta_tdx
+\lambda\int(\divv u)^2(u\cdot\nabla\theta)dx\nonumber\\
&=\lambda\frac{d}{dt}\int(\divv u)^2\theta dx-2\lambda\int\theta\divv u
\divv\dot{u}dx
+2\lambda\int\theta\divv u\divv(u\cdot\nabla u)dx
+\lambda\int(\divv u)^2(u\cdot\nabla\theta)dx\nonumber\\
&=\lambda\frac{d}{dt}\int(\divv u)^2\theta dx-2\lambda\int\theta\divv u
\divv \dot{u}dx
+2\lambda\int\theta\divv u\partial_iu^j\partial_ju^idx+\lambda\int u\cdot\nabla(\theta(\divv u)^2)dx\nonumber\\
&=\lambda\frac{d}{dt}\int(\divv u)^2\theta dx-2\lambda\int\theta\divv u
\divv \dot{u}dx
+2\lambda\int\theta{\rm div}u\partial_iu^j\partial_ju^idx-\lambda\int\theta(\divv u)^3dx\nonumber\\
&\le \lambda\frac{d}{dt}\int(\divv u)^2\theta dx
+\delta\big(\|\nabla u\|_{L^4}^4+\|\nabla\dot{u}\|_{L^2}^2\big)+C(\delta)\|\theta\|_{L^6}\|\nabla u\|_{L^2}^\frac13\|\nabla u\|_{L^4}^\frac23
\|\nabla u\|_{L^2}^2\nonumber\\
&\le \lambda\frac{d}{dt}\int(\divv u)^2\theta dx
+\delta\|\nabla\dot{u}\|_{L^2}^2+C\big(1+\|\nabla u\|_{L^4}^4+\|\nabla\theta\|_{L^2}^4\big).
\end{align}
Similarly to \eqref{f25}, we can also get that
\begin{align}
J_4&\le \frac{\mu}{2}\frac{d}{dt}\int\mathcal{D}:\mathcal{D}\theta dx
+\delta\|\nabla\dot{u}\|_{L^2}^2+C\big(1+\|\nabla u\|_{L^4}^4+\|\nabla\theta\|_{L^2}^4\big).
\end{align}
For the term $J_5$, we deduce from Gagliardo-Nirenberg, H\"older's, and Young's inequalities that
\begin{align}
J_5&=4\mu_r\int\Big|\frac12\curl u-w\Big|^2\theta_tdx
+4\mu_r\int\Big|\frac12\curl u-w\Big|^2u\cdot\nabla\theta dx\nonumber\\
&=4\mu_r\frac{d}{dt}\int\Big|\frac12\curl u-w\Big|^2\theta dx
-8\mu_r\int\Big(\frac12\curl u-w\Big)\cdot\Big(\frac12
\curl u_t-w_t\Big)\theta dx\nonumber\\
&\quad+4\mu_r\int\Big|\frac12\curl u-w\Big|^2u\cdot\nabla\theta dx\nonumber\\
&=4\mu_r\frac{d}{dt}\int\Big|\frac12\curl u-w\Big|^2\theta dx
+4\mu_r\int\Big|\frac12\curl u-w\Big|^2u\cdot\nabla\theta dx\nonumber\\
&\quad-8\mu_r\int\Big(\frac12\curl u-w\Big)\cdot\Big(\frac12\curl (\dot{u}-u\cdot\nabla u)
-(\dot{w}-u\cdot\nabla w)\Big)\theta dx\nonumber\\
&\le 4\mu_r\frac{d}{dt}\int\Big|\frac12\curl u-w\Big|^2\theta dx
-4\mu_r\int \divv u\theta\Big|\frac12\curl u-w\Big|^2dx\nonumber\\
&\quad-8\mu_r\int\Big(\frac12\curl u-w\Big)\cdot\Big(\frac12\curl \dot{u}-\dot{w}\Big)\theta dx+C\int\theta|\nabla u|^2(|\curl u|+|w|)dx\nonumber\\
&\le 4\mu_r\frac{d}{dt}\int\Big|\frac12\curl u-w\Big|^2\theta dx+\delta\big(\|\nabla\dot{u}\|_{L^2}^2+\|\dot{w}\|_{L^2}^2\big)
+C(\delta)\|\theta w\|_{L^2}^2\nonumber\\
&\quad+C(\delta)\|\theta\nabla u\|_{L^2}^2
+C(\delta)\big(\|\nabla u\|_{L^4}^4+\|w\|_{L^4}^4\big)\nonumber\\
&\le 4\mu_r\frac{d}{dt}\int\Big|\frac12\curl u-w\Big|^2\theta dx+\delta\big(\|\nabla\dot{u}\|_{L^2}^2+\|\dot{w}\|_{L^2}^2\big)
+C\|\theta\|_{L^6}^2\|w\|_{L^2}^\frac23\|w\|_{L^4}^\frac43\nonumber\\
&\quad+C\|\theta\|_{L^6}^2\|\nabla u\|_{L^2}^\frac23\|\nabla u\|_{L^4}^\frac43
+C(\delta)\big(\|\nabla u\|_{L^4}^4+\|w\|_{L^4}^4\big)\nonumber\\
&\le 4\mu_r\frac{d}{dt}\int\Big|\frac12\curl u-w\Big|^2\theta dx+C\big(1+\|\nabla\theta\|_{L^2}^4+\|\nabla u\|_{L^4}^4+\|w\|_{L^4}^4\big).
\end{align}
Similarly to $J_3$ and $J_4$, it follows from the direct calculations and H\"older's inequality that
\begin{align*}
J_6&\le c_0\frac{d}{dt}\int(\divv w)^2\theta dx
+\delta\big(\|\nabla w\|_{L^4}^4+\|\nabla\dot{w}\|_{L^2}^2\big)+C(\delta)\big(1+\|\nabla\theta\|_{L^2}^4
+\|\nabla u\|_{L^4}^4\big),\\
J_7&\le (c_a+c_d)\frac{d}{dt}\int|\nabla w|^2\theta dx+\delta\big(\|\nabla\dot{w}\|_{L^2}^2
+\|\nabla w\|_{L^4}^4\big)+C(\delta)\big(1+\|\nabla\theta\|_{L^2}^4+\|\nabla u\|_{L^4}^4\big).
\end{align*}
By H\"older's inequality, we get that
\begin{align*}
J_8&=(c_d-c_a)\frac{d}{dt}\int\partial_iw^j\partial_jw^i\theta dx
-(c_d-c_a)\int\big(\partial_i\dot{w}^j\partial_jw^i+\partial_iw^j\partial_j\dot{w}^i\big)
\theta dx\nonumber\\
&\quad+(c_d-c_a)\int\big(\partial_iu^k\partial_kw^j\partial_jw^i+\partial_ju^k
\partial_kw^i\partial_iw^j
-{\rm div}u\partial_iw^j\partial_jw^i\big)\theta dx\nonumber\\
&\le (c_d-c_a)\frac{d}{dt}\int\partial_iw^j\partial_jw^i\theta dx+\delta\big(\|\nabla\dot{w}\|_{L^2}^2+\|\nabla w\|_{L^4}^4\big)\nonumber\\
&\quad+\delta^{-1}\|\theta\|_{L^6}\|\nabla w\|_{L^2}^\frac13\|\nabla w\|_{L^4}^\frac23
+\delta^{-\frac32}\|\theta\|_{L^6}\|\nabla u\|_{L^2}^\frac13\|\nabla u\|_{L^4}^\frac23\nonumber\\
&\le (c_d-c_a)\frac{d}{dt}\int\partial_iw^j\partial_jw^i\theta dx
+\delta\|\nabla\dot{w}\|_{L^2}^2+C\big(1+\|\nabla\theta\|_{L^2}^4+\|\nabla u\|_{L^4}^4+\|\nabla w\|_{L^4}^4\big).
\end{align*}
Combining all the estimates for $J_i\ (i=1, 2, \cdots, 8)$ altogether and choosing $\eta_1\le \delta$ suitably small, one obtains that
\begin{align}\label{4.11}
&\frac{d}{dt}\int\Phi dx+c_v\|\sqrt{\rho}\dot{\theta}\|_{L^2}^2\nonumber\\
&\le C+C\|\nabla\theta\|_{L^2}^4+C\eta_1\|\nabla\dot{u}\|_{L^2}^2
+C\eta_1\|\nabla\dot{w}\|_{L^2}^2
+\|\nabla u\|_{L^4}^4+C\|\nabla w\|_{L^4}^4+C\|w\|_{L^4}^4,
\end{align}
where
\begin{align}\label{4.12}
\Phi&\triangleq\kappa|\nabla\theta|^2-2\lambda(\divv u)^2\theta-4\mu \mathcal{D}:\mathcal{D}\theta
-8\mu_r\Big|\frac12\curl u-w\Big|^2\theta
-2c_0(\divv w)^2\theta\nonumber\\
&\quad-2(c_a+c_d)|\nabla w|^2\theta-2(c_d-c_a)\partial_iw^j\partial_jw^i\theta.
\end{align}

3. Adding \eqref{4.11} to \eqref{4.5} multiplied by $\eta_1^\frac12$, we then deduce after choosing $\eta_1$ suitably small that
\begin{align}\label{4.13}
&2\frac{d}{dt}\int\Big(\Phi+\eta_1^\frac12\rho|\dot{u}|^2
+j_I\eta_1^\frac12\rho|\dot{w}|^2\Big)dx
+c_v\|\sqrt{\rho}\dot{\theta}\|_{L^2}^2
+\frac{\mu\eta_1^\frac12}{2}\|\nabla\dot{u}\|_{L^2}^2
+(\mu+\lambda)\eta_1^\frac12\|\divv\dot{u}\|_{L^2}^2\nonumber\\
&\quad+\frac{c_d\eta_1^\frac12}{2}\|\nabla\dot{w}\|_{L^2}^2
+(c_0+c_d)\eta_1^\frac12\|\divv\dot{w}\|_{L^2}^2
+c_a\eta_1^\frac12\|\curl \dot{w}\|_{L^2}^2+4\eta_1^\frac12\mu_r\Big\|\frac12\curl
\dot{u}-\dot{w}\Big\|_{L^2}^2\nonumber\\
&\le C+C\|\nabla\theta\|_{L^2}^4+C\|\nabla u\|_{L^4}^4+C\|\nabla w\|_{L^4}^4+C\|w\|_{L^4}^4,
\end{align}
where $\Phi$ is given by \eqref{4.12}. According to Lemma \ref{l23}, we have
\begin{align}
\|\nabla u\|_{L^6}&\le C\big(\|\rho\dot{u}\|_{L^2}+\|\nabla w\|_{L^2}+\|\rho\theta\|_{L^6}\big)
\le C\big(\|\sqrt{\rho}\dot{u}\|_{L^2}+\|\nabla w\|_{L^2}+\|\nabla\theta\|_{L^2}\big),\label{4.14}\\
\|\nabla w\|_{L^6}&\le C\big(\|\rho\dot{w}\|_{L^2}+\|\nabla u\|_{L^2}+\|w\|_{L^2}\big)
\le C\big(\|\sqrt{\rho}\dot{w}\|_{L^2}+\|\nabla u\|_{L^2}+\|w\|_{L^2}\big).\label{4.15}
\end{align}
Hence, by the Young and Gagliardo-Nirenberg inequalities, one derives from \eqref{4.14} and \eqref{4.15} that
\begin{align}
\|\nabla u\|_{L^4}^4+\|\nabla w\|_{L^4}^4&\le C\|\nabla u\|_{L^2}\|\nabla u\|_{L^6}^3
+C\|\nabla w\|_{L^2}\|\nabla w\|_{L^6}^3\nonumber\\
&\le C(\|\nabla u\|_{L^2}^4+\|\nabla w\|_{L^2}^4)+C\|\nabla u\|_{L^6}^4+C\|\nabla w\|_{L^6}^4\nonumber\\
&\le C(\|\nabla u\|_{L^2}^4+\|\nabla w\|_{L^2}^4+\|\nabla\theta\|_{L^2}^4+\|w\|_{L^2}^4)
+C\|\sqrt{\rho}\dot{u}\|_{L^2}^4+C\|\sqrt{\rho}\dot{w}\|_{L^2}^4\nonumber\\
&\le C\|\nabla\theta\|_{L^2}^4+C\|\sqrt{\rho}\dot{u}\|_{L^2}^4
+C\|\sqrt{\rho}\dot{w}\|_{L^2}^4+C,
\end{align}
which together with \eqref{4.13} implies that
\begin{align}\label{4.17}
&2\frac{d}{dt}\int\Big(\Phi+\eta_1^\frac12\rho|\dot{u}|^2
+j_I\eta_1^\frac12\rho|\dot{w}|^2\Big)dx
+c_v\|\sqrt{\rho}\dot{\theta}\|_{L^2}^2
+\frac{\mu\eta_1^\frac12}{2}\|\nabla\dot{u}\|_{L^2}^2
+(\mu+\lambda)\eta_1^\frac12\|\divv\dot{u}\|_{L^2}^2\nonumber\\
&\quad+\frac{c_d\eta_1^\frac12}{2}\|\nabla\dot{w}\|_{L^2}^2
+(c_0+c_d)\eta_1^\frac12\|\divv\dot{w}\|_{L^2}^2
+c_a\eta_1^\frac12\|\curl \dot{w}\|_{L^2}^2+4\eta_1^\frac12\mu_r\Big\|\frac12\curl
\dot{u}-\dot{w}\Big\|_{L^2}^2\nonumber\\
&\le C\|\nabla\theta\|_{L^2}^4+C\|\sqrt{\rho}\dot{u}\|_{L^2}^4
+C\|\sqrt{\rho}\dot{w}\|_{L^2}^4+C\|w\|_{L^4}^4+C.
\end{align}
By the definition of $\Phi$, H\"older's, Young's, and Gagliardo-Nirenberg inequalities, we get from Corollary \ref{c31}, \eqref{4.14}, and \eqref{4.15} that
\begin{align}\label{4.18}
&2\int\Big(\Phi+\eta_1^\frac12\rho|\dot{u}|^2
+j_I\eta_1^\frac12\rho|\dot{w}|^2\Big)dx\nonumber\\
&\ge 2\kappa\|\nabla\theta\|_{L^2}^2
-C\|\theta\|_{L^6}\|\nabla u\|_{L^2}^\frac32\|\nabla u\|_{L^6}^\frac12
-C\|\theta\|_{L^6}\|\nabla w\|_{L^2}^\frac32\|\nabla w\|_{L^6}^\frac12\nonumber\\
&\quad-C\|\theta\|_{L^6}\|w\|_{L^2}^\frac32\|\nabla w\|_{L^2}^\frac12
+2\int\Big(\eta_1^\frac12\rho|\dot{u}|^2+j_I\eta_1^\frac12\rho|\dot{w}|^2\Big)dx
\nonumber\\
&\ge \frac32\kappa\|\nabla\theta\|_{L^2}^2
-C(1+\|\nabla u\|_{L^6}+\|\nabla w\|_{L^6})+
2\int\Big(\eta_1^\frac12\rho|\dot{u}|^2+j_I\eta_1^\frac12\rho|\dot{w}|^2\Big)dx
\nonumber\\
&\ge \frac32\kappa\|\nabla\theta\|_{L^2}^2
-C(1+\|\sqrt{\rho}\dot{u}\|_{L^2}+\|\sqrt{\rho}\dot{w}\|_{L^2})+
2\int\Big(\eta_1^\frac12\rho|\dot{u}|^2+j_I\eta_1^\frac12\rho|\dot{w}|^2\Big)dx
\nonumber\\
&\ge \kappa\|\nabla\theta\|_{L^2}^2
-C(\eta_1)+\int\Big(\eta_1^\frac12\rho|\dot{u}|^2
+j_I\eta_1^\frac12\rho|\dot{w}|^2\Big)dx.
\end{align}
Thus, integrating \eqref{4.17} over $[0, T]$, we deduce the desired \eqref{4.1} from \eqref{4.18}, Corollary \ref{c31}, and Gronwall's inequality.
\end{proof}

\begin{lemma}\label{l42}
Let $\varepsilon_0$ be as in Lemma \ref{l37} and assume that $N_0\le \varepsilon_0$, then it holds that
\begin{align}\label{4.19}
\sup_{0\le t\le T}\big(\|\sqrt{\rho}\dot{\theta}\|_{L^2}^2+\|\nabla^2\theta\|_{L^2}^2\big)
+\int_0^T\|\nabla\dot{\theta}\|_{L^2}^2dt\le C(T).
\end{align}
\end{lemma}
\begin{proof}[Proof]

1. Applying the operator $\partial_t+\divv(u\cdot)$ to $\eqref{a1}_4$ gives rise to
\begin{align}\label{t8}
c_v \rho\big(\dot{\theta}_{t}+u \cdot \nabla\dot{\theta}\big)=& \kappa \Delta \dot{\theta}+\kappa\big(\divv u \Delta \theta-\partial_{i}(\partial_{i} u \cdot \nabla \theta)-\partial_{i} u \cdot \nabla \partial_{i} \theta\big)\nonumber \\
&+\big(\lambda(\divv  u)^{2}+2 \mu \mathcal{D}: \mathcal{D}\big) \divv u+R \rho \theta \partial_{k} u^{i} \partial_{i} u^{k}-R \rho \dot{\theta} {\rm div}u-R \rho \theta \divv \dot{u}\nonumber \\
&+2 \lambda(\divv \dot{u}-\partial_{k} u^{i} \partial_{i} u^{k}) {\rm div}u+\mu(\partial_{i} u^{j}+\partial_{j} u^{i})\big(\partial_{i} \dot{u}^{j}+\partial_{j} \dot{u}^{i}-\partial_{i} u^{k} \partial_{k} u^{j}-\partial_{j} u^{k} \partial_{k} u^{i}\big)\nonumber  \\
&+8 \mu_{r}\Big(\frac{1}{2} \curl u-w\Big)\cdot\Big(\frac{1}{2} \curl \dot{u}-\dot{w}\Big)+4 \mu_{r}\divv u\Big|\frac12\curl u-w\Big|^{2} \nonumber \\
&+8 \mu_{r}\Big(\frac{1}{2} \curl u-w\Big)\cdot\Big(\frac{1}{2} u \cdot \nabla \curl u-\frac{1}{2}\curl(u \cdot \nabla u)\Big)\nonumber  \\
&+(c_{0}(\divv w)^{2}+(c_{a}+c_{d}) \nabla w: \nabla w^{\top}+(c_{d}-c_{a}) \nabla w: \nabla w) \divv u\nonumber  \\
&+2 c_{0}(\divv \dot{w}-\partial_{k} u^{i} \partial_{i} w^{k})\divv w+2(c_{a}+c_{d})(\partial_{i} \dot{w}^{j}-\partial_{i} u^{k} \partial_{k} w^{j}) \partial_{i} w^{j}\nonumber  \\
&+(c_{d}-c_{a})\big(\big(\partial_{i} \dot{w}^{j}-\partial_{i} u^{k} \partial_{k} w^{j}\big) \partial_{j} w^{i}+\big(\partial_{j} \dot{w}^{i}-\partial_{j} u^{k} \partial_{k} w^{i}\big) \partial_{i} w^{j}\big).
\end{align}
Multiplying \eqref{t8} by $\dot{\theta}$ and integrating the resulting equation on $\mathbb{R}^3$, we get from H{\"o}lder's inequality and Sobolev's inequality that
\begin{align}\label{t9}
&\frac{c_v}{2}\frac{d}{dt}\int\rho\dot{\theta}dx+\kappa\int|\nabla \dot{\theta}|^2dt \nonumber\\
&\le \int |\nabla u|\big(|\nabla^{2} \theta||\dot{\theta}|+|\nabla \theta \| \nabla \dot{\theta}|\big)dx+C\int|\nabla u|^2|\dot{\theta}|\big(|\nabla u|+\rho\theta\big)dx\nonumber\\
&\quad+C\int \rho|\dot{\theta}|^{2}|\nabla u| dx dt+C\int \rho \theta|\nabla \dot{u} \| \dot{\theta}| dx dt \nonumber\\
&\quad+C\int |\nabla u\|\nabla \dot{u}\| \dot{\theta}| dx dt+C \int|\nabla u\|\dot{\theta}\| \dot{w}| dx dt
+C\int |w||\nabla \dot{u} \| \dot{\theta}| dx dt \nonumber\\
&\quad+C\int |w\|\dot{w}\| \dot{\theta}| dx dt+C\int |w||\nabla u|^{2}|\dot{\theta}| dx dt+C\int |\nabla u \| w|^{2}|\dot{\theta}|dx dt\nonumber \\
&\quad+C\int |\nabla w|^{2}|\nabla u|| \dot{\theta}|dx dt+C\int|\nabla \dot{w}||\nabla w||\dot{\theta}|dx dt
\nonumber\\
&\le C\|\nabla u\|_{L^3}\big(\|\nabla^2\theta\|_{L^2}\|\dot{\theta}\|_{L^6}
+\|\nabla\dot{\theta}\|_{L^2}\|\nabla\theta\|_{L^6}\big)
+C\|\rho\theta\|_{L^3}\|\nabla u\|_{L^4}^2\|\dot{\theta}\|_{L^6}\nonumber\\
&\quad+C\|\nabla u\|_{L^3}\|\nabla u\|_{L^4}^2\|\dot{\theta}\|_{L^6}
+C\|\sqrt{\rho}\dot{\theta}\|_{L^2}\|\dot{\theta}\|_{L^6}\|\nabla u\|_{L^3}
+C\|\rho\theta\|_{L^3}\|\nabla\dot{u}\|_{L^2}\|\dot{\theta}\|_{L^6}\nonumber\\
&\quad+C\|\nabla u\|_{L^3}\|\nabla\dot{u}\|_{L^2}\|\dot{\theta}\|_{L^6}
+C\|w\|_{L^3}\|\nabla\dot{u}\|_{L^2}\|\dot{\theta}\|_{L^6}
+C\|w\|_{L^3}\|\dot{w}\|_{L^2}\|\dot{\theta}\|_{L^6}\nonumber\\
&\quad+C\|w\|_{L^3}\|\nabla u\|_{L^4}^2\|\dot{\theta}\|_{L^6}
+C\|\dot{\theta}\|_{L^6}\|\nabla u\|_{L^2}\|w\|_{L^6}^2
+C\|\nabla w\|_{L^4}^2\|\nabla u\|_{L^3}\|\dot{\theta}\|_{L^6}\nonumber\\
&\quad+C\|\nabla\dot{w}\|_{L^2}\|\nabla w\|_{L^3}\|\dot{\theta}\|_{L^6}\nonumber\\
&\le C\big(\|\sqrt{\rho}\theta\|_{L^3}^2+\|\nabla u\|_{L^3}^2+\|\nabla w\|_{L^3}^2+\|w\|_{L^3}^2\big)\big(
\|\sqrt{\rho}\dot{\theta}\|_{L^2}^2+\|\nabla^2\theta\|_{L^2}^2+\|\nabla w\|_{L^4}^4\nonumber\\
&\quad+\|\nabla\dot{w}\|_{L^2}^2
+\|\nabla u\|_{L^4}^4+\|\nabla\dot{u}\|_{L^2}^2\big)
+\frac{\kappa}{2}\|\nabla\dot{\theta}\|_{L^2}^2,
\end{align}
which leads to
\begin{align}\label{4.22}
& c_v\frac{d}{dt}\|\sqrt{\rho}\dot{\theta}\|_{L^2}^2
+\kappa\|\nabla\dot{\theta}\|_{L^2}^2 \notag \\
&\le C\big(\|\sqrt{\rho}\theta\|_{L^3}^2+\|\nabla u\|_{L^3}^2+\|\nabla w\|_{L^3}^2+\|w\|_{L^3}^2\big)\nonumber\\
&\quad \times \big( \|\sqrt{\rho}\dot{\theta}\|_{L^2}^2+\|\nabla^2\theta\|_{L^2}^2
+\|\nabla w\|_{L^4}^4+\|\nabla\dot{w}\|_{L^2}^2+\|\dot{w}\|_{L^2}^2
+\|\nabla u\|_{L^4}^4+\|\nabla\dot{u}\|_{L^2}^2\big).
\end{align}

2. It follows from H{\"o}lder's inequality, Sobolev's inequality, \eqref{4.14}, \eqref{4.15}, \eqref{4.1}, and Corollary \ref{c31} that
\begin{align}\label{4.23}
&\sup_{0\le t\le T}\big(\|\sqrt{\rho}\theta\|_{L^3}^2+\|\nabla u\|_{L^3}^2+\|\nabla w\|_{L^3}^2+\|w\|_{L^3}^2\big)\nonumber\\
&\le C\sup_{0\le t\le T}\big(\|\sqrt{\rho}\theta\|_{L^2}^2+
\|\nabla \theta\|_{L^2}^2+\|\nabla u\|_{L^2}^2+\|w\|_{H^1}^2+\|\nabla u\|_{L^6}^2+\|\nabla w\|_{L^6}^2\big)\nonumber\\
&\le C\sup_{0\le t\le T}\big(\|\sqrt{\rho}\theta\|_{L^2}^2+
\|\nabla \theta\|_{L^2}^2+\|\nabla u\|_{L^2}^2+\|\sqrt{\rho}\dot{w}\|_{L^2}^2
+\|\sqrt{\rho}\dot{u}\|_{L^2}^2+\|w\|_{H^1}^2\big)\nonumber\\
&\le C(T).
\end{align}
From  \eqref{4.1} and Corollary \ref{c31}, we find that
\begin{align}\label{4.25}
&\int_0^T\big(\|\nabla^2\theta\|_{L^2}^2+\|\nabla w\|_{L^4}^4+\|\nabla\dot{w}\|_{L^2}^2
+\|\nabla u\|_{L^4}^4+\|\nabla\dot{u}\|_{L^2}^2\big)dt\nonumber\\
&\le C\int_0^T\big(\|\sqrt{\rho}\dot{\theta}\|_{L^2}^2
+\|\nabla u\|_{L^4}^4
+\|w\|_{L^4}^4+\|\nabla w\|_{L^4}^4+\|\nabla\dot{w}\|_{L^2}^2+\|\dot{w}\|_{L^2}^2
+\|\nabla u\|_{L^4}^4+\|\nabla\dot{u}\|_{L^2}^2\big)dt\nonumber\\
&\le C+C\int_0^T\big(\|\nabla w\|_{L^4}^4+\|\nabla u\|_{L^4}^4\big)dt\nonumber\\
&\le C+C\int_0^T\big(\|\nabla\theta\|_{L^2}^4+\|\sqrt{\rho}\dot{u}\|_{L^2}^4
+\|\sqrt{\rho}\dot{w}\|_{L^2}^4\big)dt\le C(T).
\end{align}
Consequently, integrating \eqref{4.22} in $t$ together with \eqref{4.23}, \eqref{4.25}, and \eqref{1.7}$_3$ yields \eqref{4.19}.
\end{proof}

\begin{lemma}\label{lz}
Let $q$ be as in Theorem \ref{thm1}.
Let $\varepsilon_0$ be as in Lemma \ref{l37} and assume that $N_0\le \varepsilon$, then it holds that
\begin{align}
\sup_{0\le t\le T}\big(\|\nabla\rho\|_{L^2\cap L^q}+\|\nabla^2u\|_{L^2}+\|\nabla^2w\|_{L^2}\big)
+\int_0^T\big(\|\nabla^2u\|_{L^q}^2+\|\nabla^2\theta\|_{L^q}^2
+\|\nabla^2w\|_{L^q}^2\big)dt\le C(T).
\end{align}
\end{lemma}
\begin{proof}[Proof]
1. Taking spatial derivative $\nabla$ on the mass equation \eqref{a1}$_1$
leads to
\begin{equation*}
\partial_{t}\nabla\rho+u\cdot\nabla^2\rho
+\nabla u\cdot\nabla\rho+\divv u\nabla\rho+\rho\nabla\divv u=0.
\end{equation*}
For $q\in (3, 6)$, multiplying the above equality by $q|\nabla\rho|^{q-2}\nabla\rho$ gives that
\begin{align*}
(|\nabla\rho|^q)_t+\divv(|\nabla\rho|^qu)+(q-1)|\nabla\rho|^q\divv u+q|\nabla\rho|^{q-2}(\nabla\rho)^{tr}\nabla u(\nabla\rho)+q\rho|\nabla\rho|^{q-2}\nabla\rho\cdot\nabla\divv u=0.
\end{align*}
Thus, integration by parts over $\mathbb{R}^3$ yields that
\begin{align}\label{0t86}
\frac{d}{dt}\|\nabla\rho\|_{L^q}&\le C\big(\|\nabla u\|_{L^\infty}+1\big)\|\nabla\rho\|_{L^q}+\|\nabla^2u\|_{L^q}.
\end{align}
Applying the standard $L^q$-estimate for the elliptic system
\begin{align*}
\begin{cases}
(\mu+\mu_r)\Delta u+(\mu+\lambda-\mu_r)\nabla\divv u=\rho\dot{u}+\nabla p
+2\mu_r\curl w, \\
u\rightarrow 0~{\rm as}~|x|\rightarrow\infty,
\end{cases}
\end{align*}
we have that
\begin{align}\label{t87}
\|\nabla^2u\|_{L^q}&\le C\big(\|\rho\dot{u}\|_{L^q}+\|\nabla (\rho\theta)\|_{L^q}+\|\curl w\|_{L^q}\big)\nonumber\\
&\le C\big(\|\rho\dot{u}\|_{L^2}+\|\rho\dot{u}\|_{L^6}
+\|\rho\|_{L^\infty}\|\nabla\theta\|_{L^q}
+\|\theta\|_{L^\infty}\|\nabla\rho\|_{L^q}
+\|w\|_{H^2}\big)\nonumber\\
&\le C\big(1+\|\nabla\dot{u}\|_{L^2}+\|\nabla\rho\|_{L^q}+\|w\|_{H^2}\big),
\end{align}
due to Sobolev's inequality, Corollary \ref{c31}, \eqref{4.1}, \eqref{4.19}, and Gagliardo-Nirenberg inequality. Thus, we get from \eqref{0t86} and \eqref{t87} that
\begin{align}\label{t86}
\frac{d}{dt}\|\nabla\rho\|_{L^q}
\le C\big(\|\nabla u\|_{L^\infty}+1\big)\|\nabla\rho\|_{L^q}
+C\big(1+\|\nabla\dot{u}\|_{L^2}
+\|\nabla\rho\|_{L^q}+\|w\|_{H^2}\big).
\end{align}

2. It follows from Gagliardo-Nirenberg inequality, Lemma \ref{l23}, and \eqref{4.1} that
\begin{align}\label{t88}
& \|\divv u\|_{L^\infty}+\|\curl u\|_{L^\infty} \notag \\
& \le C\big(\|F_1\|_{L^\infty}+\|\rho\theta\|_{L^\infty}
+\|\curl u\|_{L^\infty}\big)\nonumber\\
&\le C\|F_1\|_{L^6}^\frac12\|\nabla F_1\|_{L^6}^\frac12
+C\|\nabla\theta\|_{L^2}^\frac12\|\nabla^2\theta\|_{L^2}^\frac12
+C\|\curl u\|_{L^6}^\frac12\|\nabla\curl w\|_{L^6}^\frac12\nonumber\\
&\le C\big(1+\|\nabla F_1\|_{L^2}+\|\nabla F_1\|_{L^6}
+\|\curl u\|_{L^6}+\|\nabla\curl u\|_{L^6}\big)\nonumber\\
&\le C\big(1+\|\nabla\dot{u}\|_{L^2}+\|w\|_{H^2}\big),
\end{align}
which together with Lemma \ref{l24} and \eqref{t87} yields that
\begin{align}\label{t89}
\|\nabla u\|_{L^\infty}&\le C\big(\|\divv u\|_{L^\infty}
+\|\curl u\|_{L^\infty}\big)\ln\big(e+\|\nabla^2u\|_{L^q}\big)
+C\|\nabla u\|_{L^2}+C\nonumber\\
&\le C\big(1+\|\nabla\dot{u}\|_{L^2}+\|w\|_{H^2}\big)
\ln\big(1+\|\nabla\dot{u}\|_{L^2}+\|\nabla\rho\|_{L^q}
+\|w\|_{H^2}\big)+C.
\end{align}
Substituting \eqref{t89} into \eqref{t86} leads to
\begin{align}\label{t90}
f'(t)\leq Cg(t)f(t)\ln f(t),
\end{align}
where
\begin{align*}
f(t)\triangleq 1+\|\nabla\rho\|_{L^q},\ g(t)\triangleq 1+\|\nabla\dot{u}\|_{L^2}+\|w\|_{H^2}.
\end{align*}
Then we derive from \eqref{t90}, Gronwall's inequality, Corollary \ref{c31}, and \eqref{4.1} that
\begin{align}\label{t91}
\sup_{0\le t\le T}\|\nabla\rho\|_{L^q}\le C(T).
\end{align}
It should be noted that \eqref{0t86} also holds true for $q=2$, hence we obtain from Gronwall's inequality, \eqref{t89}, \eqref{t91}, \eqref{t87},
\eqref{4.1}, and Corollary \ref{c31} that
\begin{align}\label{zt91}
\sup_{0\le t\le T}\|\nabla\rho\|_{L^2}\le C(T).
\end{align}

3. We infer from \eqref{4.1}, \eqref{zt91}, Corollary \ref{c31}, and \eqref{4.19} that
\begin{align}\label{4.34}
\sup_{0\le t\le T}\|\nabla^2u\|_{L^2}&\le C\sup_{0\le t\le T}\big(\|\rho\dot{u}\|_{L^2}
+\|\nabla(\rho\theta)\|_{L^2}+\|\nabla w\|_{L^2}\big)\nonumber\\
&\le C\sup_{0\le t\le T}
\big(\|\rho\|_{L^\infty}^{\frac12}\|\sqrt{\rho}\dot{u}\|_{L^2}
+\|\theta\|_{L^\infty}\|\nabla\rho\|_{L^2}+\|\rho\|_{L^\infty}
\|\nabla\theta\|_{L^2}+\|\nabla w\|_{L^2}\big) \notag \\
&\le C(T), \\
\sup_{0\le t\le T}\|\nabla^2w\|_{L^2}&\le C\sup_{0\le t\le T}\big(\|\sqrt{\rho}\dot{w}\|_{L^2}
+\|\nabla u\|_{L^2}+\| w\|_{L^2}\big)\le C(T),
\end{align}
and
\begin{align}
\int_0^T\|\nabla^2u\|_{L^q}^2dt\le C\int_0^T\big(1+\|\nabla\dot{u}\|_{L^2}^2+\|\nabla\rho\|_{L^q}^2
+\|w\|_{H^2}^2\big)dt\le C(T).
\end{align}
Notice that
\begin{align}\label{4.37}
\int_0^T\|\nabla^2w\|_{L^q}^2dt
& \le C\int_0^T\big(\|\rho\dot{w}\|_{L^q}^2+\|\nabla u\|_{L^q}^2+\|w\|_{L^q}^2\big)dt\nonumber\\
& \le C\int_0^T\big(1+\|\nabla\dot{w}\|_{L^2}^2
+\|\nabla u\|_{H^1}^2+\|w\|_{H^1}^2\big)dt \notag \\
& \le C(T).
\end{align}
Finally, it follows from Sobolev's inequality, Gagliardo-Nirenberg inequality, Corollary \ref{c31}, \eqref{4.1}, \eqref{4.19}, \eqref{4.25}, and \eqref{4.34}--\eqref{4.37} that
\begin{align}
& \int_0^T\|\nabla^2\theta\|_{L^q}^2dt \nonumber\\
&\le C\int_0^T\big(\|\rho\dot{\theta}\|_{L^q}^2
+\|\theta\nabla u\|_{L^q}^2+\|\nabla u\|_{L^{2q}}^4
+\|\nabla w\|_{L^{2q}}^4\big)dt\nonumber\\
&\le C\int_0^T\big(\|\rho\dot{\theta}\|_{L^2}^2+\|\rho\dot{\theta}\|_{L^6}^2
+\|\theta\|_{L^\infty}^2\|\nabla u\|_{H^1}^2+\|\nabla u\|_{L^{2q}}^4
+\|\nabla w\|_{L^{2q}}^4\big)dt
\nonumber\\
&\le C\int_0^T\big(1+\|\nabla\dot{\theta}\|_{L^2}^2
+\|\nabla\theta\|_{L^2}^\frac12\|\nabla^2\theta\|_{L^2}^\frac12\|\nabla u\|_{H^1}^2
+\|\nabla u\|_{L^3}^2\|\nabla^2u\|_{L^q}^2
+\|\nabla w\|_{L^3}^2\|\nabla^2w\|_{L^q}^2\big)dt\nonumber\\
&\le C\int_0^T\big(1+\|\nabla\dot{\theta}\|_{L^2}^2+\|\nabla^2u\|_{L^q}^2+
\|\nabla^2w\|_{L^q}^2\big)dt \notag \\
&\le C(T).
\end{align}
The proof of Lemma \ref{lz} is finished.
\end{proof}

\section{Proof of Theorem \ref{thm1}}\label{sec5}
By Lemma \ref{l21}, there exists a small $T_0$ such that the Cauchy problem \eqref{a1}--\eqref{a3} has a unique strong solution $(\rho, u, w, \theta)$ in
$\mathbb{R}^3\times (0, T_0]$. Then, we use {\it a priori} estimates obtained by sections \ref{sec3} and \ref{sec4} to extend the local strong solution $(\rho, u, w, \theta)$ to
all time.

We deduce from the definition of $N_T$ and initial data \eqref{1.6}--\eqref{1.8} that
\begin{align}
N_0\le \varepsilon_0, \quad \rho_0< 4\bar{\rho},
\end{align}
Thus, there exists a $T_1\in (0, T_0]$ such that \eqref{3.58} holds for $T=T_1$.

Set
\begin{align}
T^*=\sup\{T~|~\eqref{3.58}~is~valid\}.
\end{align}
Then $T^*\ge T_1$. We next claim that
\begin{align}
T^*=\infty.
\end{align}
Suppose, by contradiction,
that $T^*<\infty$. Note that, all
the {\it a priori} estimates obtained in section \ref{sec3} are uniformly bounded for any $t<T^*$. Hence,
we define
\begin{align}
(\rho, u, w, \theta)(x, T^*)=\lim_{t\rightarrow T^*}(\rho, u, w, \theta)(x, t).
\end{align}
Furthermore, standard arguments yield that $(\rho\dot{u}, \rho\dot{w}, \rho\dot{\theta})\in C([0, T^*); L^2)$, which implies that
\begin{align}
(\rho\dot{u}, \rho\dot{w}, \rho\dot{\theta})(x, T^*)=\lim_{t\rightarrow T^*}(\rho\dot{u}, \rho\dot{w}, \rho\dot{\theta})(x, t)\in L^2.
\end{align}
Hence,
\begin{align*}
\begin{cases}
\big[(\lambda+\mu-\mu_r)\nabla\divv u+(\mu+\mu_r)\Delta u-R\nabla(\rho\theta)+2\mu_r\curl w\big]|_{t=T^*}=\sqrt{\rho}g_1(x),\\
\big[2\mu_r(\curl u-2w)+(c_0+c_d-c_a)\nabla\divv w+(c_a+c_d)\Delta w\big]|_{t=T^*}=\sqrt{\rho_0}g_2(x),\\
\big[\lambda(\divv u)^2+\frac{\mu}{2}|\nabla u+\nabla u^T|^2
+4\mu_r\big|\frac12\curl u-w\big|^2+c_0(\divv w)^2\\
+(c_a+c_d)\nabla w:\nabla w^T+(c_d-c_a)\nabla w:\nabla w+\kappa\Delta\theta\big]|_{t=T^*}=\sqrt{\rho_0}g_3(x),
\end{cases}
\end{align*}
with
\begin{align*}
& g_1(x)=
\begin{cases}
\rho^{-\frac{1}{2}}(x, T^*)(\rho\dot{u})(x, T^*), &{\rm for}~ x\in\{x|\rho(x, T^*)>0\},\\
0,  &{\rm for}~ x\in\{x|\rho(x, T^*)=0\},
\end{cases}\\
& g_2(x)=
\begin{cases}
\rho^{-\frac{1}{2}}(x, T^*)(\rho\dot{w})(x, T^*), &{\rm for}~ x\in\{x|\rho(x, T^*)>0\},\\
0,  &{\rm for}~ x\in\{x|\rho(x, T^*)=0\},
\end{cases}
\end{align*}
and
\begin{align*}
g_3(x)=
\begin{cases}
\rho^{-\frac{1}{2}}(x, T^*)\big(c_v\rho\dot{\theta}+\rho\theta\divv u\big)(x, T^*), &{\rm for}~ x\in\{x|\rho(x, T^*)>0\},\\
0,  &{\rm for}~ x\in\{x|\rho(x, T^*)=0\},
\end{cases}
\end{align*}
satisfying $(g_1, g_2, g_3)\in L^2$ due to Corollary \ref{c31} and \eqref{4.1}. Thus, $(\rho, u, w, \theta)$ satisfy the compatibility condition \eqref{1.7}. Therefore, we can extend the local strong solution beyond $T^*$ by taking $(\rho, u, w, \theta)(, T^*)$ as the initial data and applying the Lemma \ref{l21}. This contradicts the assumption of $T^*$. The proof of Theorem \ref{thm1} is complete.

\end{document}